\documentclass[10pt]{article}
\usepackage{amsthm}
\usepackage{amsmath}
\usepackage{amssymb}
\usepackage{makeidx}

\theoremstyle{definition}
\newtheorem{definition}{Definition}[section]

\newtheorem{remark}[definition]{Remark}

\theoremstyle{plain}
\newtheorem{theorem}[definition]{Theorem}
\newtheorem{lemma}[definition]{Lemma}
\newtheorem{proposition}[definition]{Proposition}
\newtheorem{corollary}[definition]{Corollary}

\numberwithin{equation}{section}

\def\N{{\mathbb N}}

\begin{document}
\title{Some Extremal Values of the Number of Congruences\\ of a Finite Lattice}
\author{J\' ulia KULIN and Claudia MURE\c SAN\thanks{Corresponding author}}
\date{\today }
\maketitle

\begin{abstract} We study the smallest, as well as the largest numbers of congruences of lattices of an arbitrary finite cardinality $n$. Continuing the work of Freese and Cz\' edli, we prove that the third, fourth and fifth largest numbers of congruences of an $n$--element lattice are: $5\cdot 2^{n-5}$ if $n\geq 5$, respectively $2^{n-3}$ and $7\cdot 2^{n-6}$ if $n\geq 6$. We also determine the structures of the $n$--element lattices having $5\cdot 2^{n-5}$, respectively $2^{n-3}$ congruences, along with the structures of their congruence lattices.

{\em Keywords}: (finite) lattice, (principal) congruence, (prime) interval, atom, (ordinal, horizontal) sum.

{\em MSC $2010$}: primary: 06B10; secondary: 06B05.\end{abstract}

\section{Introduction}
\label{introduction}

We shall use the notation ${\rm Con}(A)$ for the congruence lattice of an algebra $A$, along with other common notations, recalled in Section \ref{preliminaries} below. For any $n\in \N ^*$, let ${\rm NCL}(n)=\{|{\rm Con}(L)|\ |\ L$ is a lattice with $|L|=n\}\subset \N ^*$, ${\rm Gncl}(1,n)=\max ({\rm NCL}(n))$ and, for any $p\in \N ^*$ such that $\{k\in {\rm NCL}(n)\ |\ k<{\rm Gncl}(p,n)\}\neq \emptyset $, ${\rm Gncl}(p+1,n)=\max (\{k\in {\rm NCL}(n)\ |\ k<{\rm Gncl}(p,n)\})$. Also, for any $p\in \N ^*$, let ${\rm Lnc}(p,n)=\{L\ |\ L$ is a lattice with $|L|=n$ and $|{\rm Con}(L)|={\rm Gncl}(p,n)\}$. We investigate the elements of ${\rm NCL}(n)$, as well as the elements of the sets ${\rm Lnc}(p,n)$ and the structures of their congruence lattices.

Regarding the problems related to the present work, we mention the representation problem for lattices in the form of congruence lattices of lattices; its investigation goes back to R. P. Dilworth and was mile--stoned by Gr\" atzer and Schmidt~\cite{grasch}, Wehrung~\cite{wehrung}, R\r{u}\v{z}i\v{c}ka~\cite{ruzicka}, Gr\" atzer and Knapp~\cite{grakna},
and Plo\v s\v cica~\cite{ploscica}, and surveyed in Gr\" atzer~\cite{grafin} and Schmidt~\cite{schmidt}. A lot of results have been proved on the representation problem of two or more lattices and certain maps among them by (complete) congruences; for example, see Gr\"atzer and Schmidt~\cite{grSch1995}, Gr\" atzer and Lakser~\cite{gglakser86}, Cz\' edli~\cite{czgreprect,czgcomplcon}. Even the posets and monotone maps among them have been characterized by principal congruences of lattices; for example, see Gr\" atzer~\cite{gG13,gGasm2016,gGprincII,gGprincIII}, Gr\"atzer and Lakser~\cite{gGhL17}, and Cz\'edli~\cite{czgaleph0,czgonemap,czgprincout,czgmanymaps,czgcometic}. Finally, the above-mentioned trends, focusing on the sizes of congruence lattices, on the structures formed by congruences, and on maps among these structures, have recently met in Cz\' edli and Mure\c san~\cite{gccm}, enriching the first two trends and related even to the third one.

The problem of the existence of lattices $L$ with certain values for the triples of cardinalities $(|{\rm Con}(L)|,|{\rm Filt}(L)|,$\linebreak $|{\rm Id}(L)|)$ was raised in Mure\c san~\cite{eucard,eunoucard}, and given the denomination of {\em CFI--representability} in Cz\' edli and Mure\c san~\cite{gccm}: with the notations above, in the finite case in the present paper, we say that a triple $(k,n,n)$ is {\em CFI--representable} iff $k\in {\rm NCL}(n)$, and we say that the elements of $\{L\ |\ L$ is a lattice with $|L|=n$ and $|{\rm Con}(L)|=k\}$ {\em CFI--represent} the triple $(k,n,n)$. For its simplicity, we choose the terminology of CFI--representability over the notations above in the following sections of this paper.

Regarding the smallest values in ${\rm NCL}(n)$, in Section \ref{thesmall} below, we prove that, if $n\geq 7$, then: $\{2^j\ |\ j\in \overline{1,n-1}\}\subset {\rm NCL}(n)$, and, if $n\neq 8$, then $\overline{2,n+1}\subset {\rm NCL}(n)$.

Regarding the largest values in ${\rm NCL}(n)$: by \cite{gcze,free}, ${\rm Gncl}(1,n)=2^{n-1}$ and ${\rm Lnc}(1,n)=\{{\cal L}_n\}$; by \cite{gcze}, if $n\geq 4$, then ${\rm Gncl}(2,n)=2^{n-2}$ and ${\rm Lnc}(2,n)=\{{\cal L}_r\dotplus {\cal L}_2^2\dotplus {\cal L}_{n-r-2}\ |\ r\in \overline{1,n-2}\}$. The four largest values in ${\rm NCSL}(n)=\{|{\rm Con}(L)|\ |\ L$ is a semilattice with $|L|=n\}$ and the structures of the semilattices with these numbers of congruences have been determined in \cite{gcz}.

In Section \ref{thelarge} of the present paper, using the methods of \cite{gcze}, we prove that: if $n\geq 5$, then  $|{\rm Gncl}(3,n)|=5\cdot 2^{n-5}$ and ${\rm Lnc}(3,n)=\{{\cal L}_r\dotplus N_5\dotplus {\cal L}_{n-r-3}\ |\ r\in \overline{1,n-4}\}$; if $n\geq 6$, then $|{\rm Gncl}(4,n)|=2^{n-3}$, ${\rm Lnc}(4,n)=\{{\cal L}_r\dotplus ({\cal L}_2\times {\cal L}_3)\dotplus {\cal L}_{n-r-4},{\cal L}_s\dotplus {\cal L}_2^2\dotplus {\cal L}_t\dotplus {\cal L}_2^2\dotplus {\cal L}_{n-s-t-4}\ |\ r\in \overline{1,n-5},s,t\in \N ^*,s+t\leq n-5\}$ and $|{\rm Gncl}(5,n)|=7\cdot 2^{n-6}$. The structures of the congruence lattices of the lattices from ${\rm Lnc}(3,n)$ and ${\rm Lnc}(4,n)$ follow from the previously mentioned results. We also conjecture that ${\rm Lnc}(5,n)=\{{\cal L}_r\dotplus ({\cal L}_3\boxplus {\cal L}_5)\dotplus {\cal L}_{n-r-4},{\cal L}_r\dotplus ({\cal L}_4\boxplus {\cal L}_4)\dotplus {\cal L}_{n-r-4}\ |\ r\in \overline{1,n-5}\}$; the methods of \cite{gcze} are probably adequate for determining ${\rm Lnc}(5,n)$, as well, but we do not pursue this proof here, as it would lenghthen our paper considerably.

\section{Definitions and Notations}
\label{preliminaries}

We shall denote by $\N $ the set of the natural numbers and by $\N ^*=\N \setminus \{0\}$. $\amalg $ shall be the disjoint union of sets. For any set $M$, $|M|$ shall be the cardinality of $M$, ${\rm Eq}(M)$ the set of the equivalences on $M$, $\Delta _M=\{(x,x)\ |\ x\in M\}$ and $\nabla _M=M^2$; for any partition $\pi $ of $M$, $eq(\pi )$ shall be the equivalence on $M$ that corresponds to $\pi $; if $\pi =\{M_1,\ldots ,M_n\}$ for some $n\in \N ^*$, then $eq(\pi )$ shall simply be denoted by $eq(M_1,\ldots ,M_n)$.

All lattices shall be non--empty and, unless mentioned otherwise, they shall be designated by their underlying sets, and their operations and order relation shall be denoted in the usual way, and $\prec $ shall denote their succession relation. The {\em trivial lattice} shall be the one--element lattice. $\cong $ shall denote the existence of a lattice isomorphism.

For any lattice $L$, $({\rm Con}(L),\vee ,\cap ,\Delta _L,\nabla _L)$ shall be the bounded lattice of the congruences of $L$ and ${\rm Filt}(L)$ and ${\rm Id}(L)$ shall be the lattices of the filters and ideals of $L$, respectively. For any $a\in L$, $[a)_L$ and $(a]_L$ shall be the principal filter, respectively ideal of $L$ generated by $a$. For any $a,b\in L$, $[a,b]_L=[a)_L\cap (b]_L$ shall be the interval of $L$ bounded by $a$ and $b$, which, of course, is non--empty iff $a\leq b$; recall that $[a,b]_L$ is called a {\em prime interval} iff $a\prec b$, and it is called a {\em narrows} iff it is a prime interval such that $a$ is meet--irreducible and $b$ is join--irreducible (see \cite{gratzer},\cite{gcze}). Following \cite{gcze}, we shall denote by ${\rm con}(a,b)$ the principal congruence of $L$ generated by $(a,b)$. If $L$ has a $0$, then $At(L)$ shall be the set of the atoms of $L$.

$\dotplus $ shall be the ordinal sum and $\boxplus $ shall be the horizontal sum. Recall that, for any lattice $(L,\leq _L,1_L)$ with largest element and any lattice $(M,\leq _M,0_M)$ with smallest element, the {\em ordinal sum} of $L$ with $M$ is defined by identifying $c=1_L=0_M\in L\cap M$ and letting $L\dotplus M=((L\setminus \{c\})\amalg \{c\}\amalg (M\setminus \{c\}),\leq _L\cup \leq _M\cup \{(x,y)\ |\ x\in L,y\in M\})$. Also, for any bounded lattices $(L,\leq _L,0_L,1_L)$ and $(M,\leq _M,0_M,1_M)$ with $|L|,|M|>2$, the {\em horizontal sum} of $L$ with $M$ is defined by identifying $0=0_L=0_M,1=1_L=1_M\in L\cap M$ and letting $L\boxplus M=((L\setminus \{0,1\})\amalg \{0,1\}\amalg (M\setminus \{0,1\}),\leq _L\cup \leq _M,0,1)$. Clearly, the ordinal sum of bounded lattices is associative, while the horizontal sum is both associative and commutative. For any $n\in \N ^*$, ${\cal L}_n$ shall denote the $n$--element chain, so that ${\cal L}_2^2={\cal L}_3\boxplus {\cal L}_3$ is the four--element Boolean algebra (the {\em rhombus}), $M_3={\cal L}_3\boxplus {\cal L}_3\boxplus {\cal L}_3$ is the five--element modular non--distributive lattice (the {\em diamond}) and $N_5={\cal L}_3\boxplus {\cal L}_4$ is the five--element non--modular lattice (the {\em pentagon}).

\section{Some Constructions and Their Effect on the Cardinalities of the Sets of Congruences, Filters and Ideals}

Following \cite{gccm}, we call a triple $(\kappa ,\lambda ,\mu )$ of nonzero cardinalities {\em CFI--representable} iff there exists a lattice $L$ such that $\kappa =|{\rm Con}(L)|$, $\lambda =|{\rm Filt}(L)|$ and $\mu =|{\rm Id}(L)|$, case in which we say that $L$ {\em CFI--represents} the triple $(\kappa ,\lambda ,\mu )$. Of course, if $L$ is finite, then all its filters and all its ideals are principal, thus $|{\rm Filt}(L)|=|{\rm Id}(L)|=|L|\in \N ^*$ and $|{\rm Con}(L)|\in \N ^*$, as well. So, if a triple $(\kappa ,\lambda ,\mu )$ is CFI--represented by a finite lattice, then $\lambda =\mu $ and they equal the cardinality of that lattice, and $\kappa $ is finite, as well. Most times, we shall use the remarks in this paper without referencing them.

\begin{remark} Let $L$ be a lattice and $\theta \in {\rm Con}(L)$. Clearly, if $S$ is a sublattice of $L$, then $\theta \cap S^2\in {\rm Con}(S)$. By \cite{gratzer}, for any $a\in L$, $a/\theta $ is a convex sublattice of $L$.\end{remark}

\begin{remark} Clearly, for any $t\in \N ^*$, if the lattices $L_1,L_2,\ldots ,L_t$ CFI--represent the cardinalities triples $(\kappa _1,\lambda _1,\mu _1),\ldots ,(\kappa _t,\lambda _t,\mu _t)$, respectively, then the lattice $\displaystyle \prod _{i=1}^tL_i$ CFI--represents $\displaystyle (\prod _{i=1}^t\kappa _i,\prod _{i=1}^t\lambda _i,\prod _{i=1}^t\mu _i)$, because $\displaystyle {\rm Con}(\prod _{i=1}^tL_i)\cong \prod _{i=1}^t{\rm Con}(L_i)$, $\displaystyle {\rm Filt}(\prod _{i=1}^tL_i)\cong \prod _{i=1}^t{\rm Filt}(L_i)$ and $\displaystyle {\rm Id}(\prod _{i=1}^tL_i)\cong \prod _{i=1}^t{\rm Id}(L_i)$.\label{cgprod}\end{remark}

\begin{remark} $(1,1,1)$ is CFI--represented by ${\cal L}_1$. For any $n\in \N ^*\setminus \{1\}$, if $(k,n,n)$ is CFI--representable, then $k\in \N ^*\setminus \{1\}$. ${\cal L}_2$ CFI--represents $(2,2,2)$ and, more generally, for any $s\in \N ^*$, the Boolean algebra ${\cal L}_2^s$ CFI--represents $(2^s,2^s,2^s)$. ${\cal L}_2^2$ CFI--represents $(4,4,4)$.\end{remark}

\begin{remark} It is immediate that, for any lattices $L$ with a $1$ and $M$ with a $0$: ${\rm Con}(L\dotplus M)=\{eq(L/\alpha \cup M/\beta )\ |\ \alpha \in {\rm Con}(L),\beta \in {\rm Con}(M)\}\cong {\rm Con}(L)\times {\rm Con}(M)$, ${\rm Filt}(L\dotplus M)={\rm Filt}(M)\cup \{F\cup M\ |\ F\in {\rm Filt}(L)\}={\rm Filt}(M)\cup \{F\cup M\ |\ F\in {\rm Filt}(L)\setminus \{1\}\}\cong {\rm Filt}(M)\dotplus {\rm Filt}(L)$ and ${\rm Id}(L\dotplus M)={\rm Id}(L)\cup \{I\cup L\ |\ I\in {\rm Id}(M)\}={\rm Id}(L)\cup \{I\cup L\ |\ I\in {\rm Id}(M)\setminus \{0\}\}\cong {\rm Id}(L)\dotplus {\rm Id}(M)$, hence, if $L$ CFI--represents $(\kappa ,\lambda ,\mu )$ and $M$ CFI--represents $(\nu ,\rho ,\sigma )$, then $L\dotplus M$ CFI--represents $(\kappa \cdot \nu ,\lambda +\rho -1,\mu +\sigma -1)$.

Therefore, more generally, for any $t\in \N ^*$, if the bounded lattices $L_1,L_2,\ldots ,L_t$ CFI--represent the triples $(\kappa _1,\lambda _1,\mu _1),\ldots ,(\kappa _t,\lambda _t,\mu _t)$, respectively, then $\displaystyle \dotplus _{i=1}^tL_i$ CFI--represents $\displaystyle (\prod _{i=1}^t\kappa _i,\sum _{i=1}^t\lambda _i-t+1,\sum _{i=1}^t\mu _i-t+1)$.

In particular, for any $n\in \N ^*$, since ${\rm Con}({\cal L}_2)\cong {\cal L}_2$ and ${\cal L}_n=\dotplus _{i=1}^{n-1}{\cal L}_2$, it follows that $\displaystyle {\rm Con}({\cal L}_n)\cong \prod _{i=1}^{n-1}{\rm Con}({\cal L}_2)\cong {\cal L}_2^{n-1}$, thus ${\cal L}_n$ CFI--represents $(2^{n-1},n,n)$. So ${\cal L}_3$ CFI--represents $(4,3,3)$ and ${\cal L}_4$ CFI--represents $(8,4,4)$. Also, if $L$ CFI--represents $(\kappa ,\lambda ,\mu )$, then $L\dotplus {\cal L}_2$ CFI--represents $(2\cdot \kappa ,\lambda +1,\mu +1)$ and, more generally, for any $s\in \N $, $L\dotplus {\cal L}_2^{s+1}$ CFI--represents $(2^s\cdot \kappa ,\lambda +s,\mu +s)$.\label{cgords}\end{remark}

\begin{remark} By the above, the only triples which are CFI--represented by lattices of cardinality at most $4$ are: $(1,1,1)$, $(2,2,2)$, $(4,3,3)$, $(4,4,4)$ and $(8,4,4)$.\label{nleq4}\end{remark}

\begin{remark} Let $t\in \N ^*$, $L_1,L_2,\ldots ,L_t$ be bounded lattices, not necessarily non--trivial, which CFI--represent the triples $(\kappa _1,\lambda _1,\mu _1),\ldots ,(\kappa _t,\lambda _t,\mu _t)$, respectively. Let $L=\boxplus _{i=1}^t({\cal L}_2\dotplus L_i\dotplus {\cal L}_2)$, so that, clearly: $\displaystyle {\rm Filt}(L)=\{\{1\},L\}\cup \{F\cup \{1\}\ |\ F\in \bigcup _{i=1}^t{\rm Filt}(L_i)\}$ and $\displaystyle {\rm Id}(L)=\{\{0\},L\}\cup \{I\cup \{0\}\ |\ I\in \bigcup _{i=1}^t{\rm Id}(L_i)\}$.

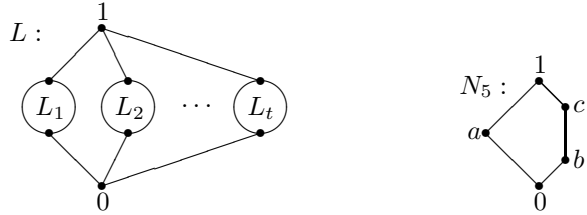
\begin{center}\begin{tabular}{cc}
\begin{picture}(60,65)(0,0)
\put(10,20){\circle*{3}}
\put(10,40){\circle*{3}}
\put(40,20){\circle*{3}}
\put(40,40){\circle*{3}}
\put(90,20){\circle*{3}}
\put(90,40){\circle*{3}}
\put(30,60){\circle*{3}}
\put(30,0){\circle*{3}}
\put(10,30){\circle{20}}
\put(40,30){\circle{20}}
\put(90,30){\circle{20}}
\put(5,27){$L_1$}
\put(35,27){$L_2$}
\put(85,27){$L_t$}
\put(60,30){$\ldots $}
\put(28,63){$1$}
\put(28,-9){$0$}
\put(-5,55){$L:$}
\put(30,0){\line(1,2){10}}
\put(30,0){\line(-1,1){20}}
\put(30,60){\line(1,-2){10}}
\put(30,60){\line(3,-1){60}}
\put(30,0){\line(3,1){60}}
\put(30,60){\line(-1,-1){20}}
\end{picture}
&\hspace*{100pt}
\begin{picture}(40,65)(0,0)
\put(20,0){\circle*{3}}
\put(20,40){\circle*{3}}
\put(0,20){\circle*{3}}
\put(30,10){\circle*{3}}
\put(30,30){\circle*{3}}
\put(18,-9){$0$}
\put(18,43){$1$}
\put(-7,18){$a$}
\put(33,7){$b$}
\put(33,28){$c$}
\put(20,0){\line(-1,1){20}}
\put(20,40){\line(1,-1){10}}
\put(20,0){\line(1,1){10}}
\put(0,20){\line(1,1){20}}
\put(20,40){\line(1,-1){10}}
\put(30,10){\line(0,1){20}}
\put(-10,35){$N_5:$}
\end{picture}\end{tabular}\end{center}\vspace*{-5pt}

By \cite{eunoucard}:\begin{itemize}
\item if $t=2$, then ${\rm Con}(L)=\{eq(\{0\}\cup L_1,\{1\}\cup L_2),eq(\{1\}\cup L_1,\{0\}\cup L_2),\nabla _L\}\cup \{eq(\{\{0\},\{1\}\}\cup L_1/\alpha \cup L_2/\beta )\ |\ \alpha \in {\rm Con}(L_1),\beta \in {\rm Con}(L_2)\}\cong ({\rm Con}(L_1)\times {\rm Con}(L_2))\dotplus {\cal L}_2^2$, hence $L$ CFI--represents $(\kappa _1\cdot \kappa _2+3,\lambda _1+\lambda _2+2,\mu _1+\mu _2+2)$; in particular, $N_5=({\cal L}_2\dotplus {\cal L}_1\dotplus {\cal L}_2)\boxplus ({\cal L}_2\dotplus {\cal L}_2\dotplus {\cal L}_2)$ CFI--represents $(5,5,5)$, since ${\rm Con}(N_5)\cong {\cal L}_2\dotplus {\cal L}_2^2$; 
\item if $t\geq 3$, then ${\rm Con}(L)=\{\nabla _L\}\cup \{eq(\{\{0\},\{1\}\}\cup L_1/\alpha _1\cup \ldots \cup L_t/\alpha _t)\ |\ (\forall \, i\in \overline{1,t})\, (\alpha _i\in {\rm Con}(L_i)\}$, hence $L$ CFI--represents $\displaystyle (1+\prod _{i=1}^t\kappa _i,2+\sum _{i=1}^t\lambda _i,2+\sum _{i=1}^t\mu _i)$; in particular, $M_3=({\cal L}_2\dotplus {\cal L}_1\dotplus {\cal L}_2)\boxplus ({\cal L}_2\dotplus {\cal L}_1\dotplus {\cal L}_2)\boxplus ({\cal L}_2\dotplus {\cal L}_1\dotplus {\cal L}_2)$ CFI--represents $(2,5,5)$.\end{itemize}

Also, in particular:\begin{itemize}
\item if $t\geq 3$ and $L_1=\ldots =L_t={\cal L}_1$, then $L=\boxplus _{i=1}^t{\cal L}_3$ CFI--represents $(2,t+2,t+2)$, hence $L\dotplus {\cal L}_2=(\boxplus _{i=1}^t{\cal L}_3)\dotplus {\cal L}_2$ CFI--represents $(4,t+3,t+3)$;
\item if $t\geq 3$, $L_1=\ldots =L_{t-1}={\cal L}_1$ and $L_t={\cal L}_2$, then $L=(\boxplus _{i=1}^{t-1}{\cal L}_3)\boxplus {\cal L}_4$ CFI--represents $(3,t+3,t+3)$, hence $L\dotplus {\cal L}_2=((\boxplus _{i=1}^{t-1}{\cal L}_3)\boxplus {\cal L}_4)\dotplus {\cal L}_2$ CFI--represents $(6,t+4,t+4)$;
\item if $t\geq 3$, $L_1=\ldots =L_{t-2}={\cal L}_1$ and $L_{t-1}=L_t={\cal L}_2$, then $L=(\boxplus _{i=1}^{t-2}{\cal L}_3)\boxplus {\cal L}_4\boxplus {\cal L}_4$ CFI--represents $(5,t+4,t+4)$;
\item if $t=2$ and $L_2={\cal L}_1$, then $L=({\cal L}_2\dotplus L_1\dotplus {\cal L}_2)\boxplus {\cal L}_3$ CFI--represents $(\kappa _1+3,\lambda _1+3,\mu _1+3)$; if we also have $L_1={\cal L}_2$, then $L\cong {\cal L}_4\boxplus {\cal L}_3\cong N_5$; let us note the congruences of the pentagon: with the elements denoted as in the Hasse diagram of the pentagon above, we have ${\rm Con}(N_5)=\{\Delta _{N_5},eq(\{0\},\{a\},\{b,c\},\{1\}),eq(\{0,b,c\},\{a,1\}),\linebreak eq(\{0,a\},\{b,c,1\}),\nabla _{N_5}\}$;
\item if $t=3$ and $L_2=L_3={\cal L}_1$, then $L=({\cal L}_2\dotplus L_1\dotplus {\cal L}_2)\boxplus {\cal L}_3\boxplus {\cal L}_3$ CFI--represents $(\kappa _1+1,\lambda _1+4,\mu _1+4)$.\end{itemize}

Note that the above also hold if we replace $t$ by an arbitrary nonzero cardinality, and, of course, the set $\overline{1,t}$ by an arbitrary non--empty set of cardinality $t$.

Thus, for all $k,n\in \N ^*$:\begin{itemize}
\item if $n\geq 5$, then $(2,n,n)$ is CFI--represented by $\boxplus _{i=1}^{n-2}{\cal L}_3$;
\item if $n\geq 6$, then $(3,n,n)$ and $(4,n,n)$ are CFI--represented by ${\cal L}_4\boxplus \boxplus _{i=1}^{n-4}{\cal L}_3$ and ${\cal L}_2\dotplus (\boxplus _{i=1}^{n-3}{\cal L}_3)$, respectively;
\item if $n\geq 7$, then $(5,n,n)$ and $(6,n,n)$ are represented by ${\cal L}_4\boxplus {\cal L}_4\boxplus \boxplus _{i=1}^{n-6}{\cal L}_3$ and ${\cal L}_2\dotplus ({\cal L}_4\boxplus \boxplus _{i=1}^{n-4}{\cal L}_3)$, respectively;
\item if $(k,n,n)$ is CFI--representable, then $(k+3,n+3,n+3)$ is CFI--representable;
\item if $(k,n,n)$ is CFI--representable, then $(k+1,n+4,n+4)$ is CFI--representable.\end{itemize}

More generally, the above hold for any nonzero cardinalities $k,n$.\label{cghs}\end{remark}

\begin{remark} If $M$ is a bounded lattice in which $0$ is meet--reducible and $L={\cal L}_3\boxplus (M\dotplus {\cal L}_2)$, then it is immediate that ${\rm Con}(L)=\{eq(M/\alpha \cup \{\{x\}\ |\ x\in L\setminus M\})\ |\ \alpha \in {\rm Con}(M)\}\cup \{eq(M,L\setminus M),\nabla _L\}$, thus $|{\rm Con}(L)|=|{\rm Con}(M)|+2$. See also \cite{eunoucard}.
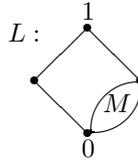
\vspace*{-5pt}

\begin{center}\begin{picture}(40,40)(0,0)
\put(20,0){\circle*{3}}
\put(20,40){\circle*{3}}
\put(0,20){\circle*{3}}
\put(40,20){\circle*{3}}
\put(18,-9){$0$}
\put(18,43){$1$}
\put(20,0){\line(-1,1){20}}
\put(20,40){\line(1,-1){20}}
\put(20,40){\line(-1,-1){20}}
\put(26,8){$M$}
\put(-10,35){$L:$}
\put(41,0){\oval(39,39)[tl]}
\put(21,20){\oval(39,39)[br]}\end{picture}\end{center}
\label{0red}\end{remark}

\section{On the Smallest Numbers of Congruences of Finite Lattices}
\label{thesmall}

\begin{proposition} Let $n\in \N $ such that $n\geq 7$. Then:\begin{enumerate}
\item\label{listcfirepr1} for any $j\in \overline{1,n-1}$, $(2^{j},n,n)$ is CFI--representable;
\item\label{listcfirepr2} if $n\neq 8$, then, for any $k\in \overline{2,n+1}$, $(k,n,n)$ is CFI--representable.
\end{enumerate}\label{listcfirepr}\end{proposition}

\begin{proof} (\ref{listcfirepr1}) For all $n\geq 8\geq 5$, $(2,n,n)$ is CFI--representable and, if $j\in \N ^*$ is such that a lattice $L$ CFI--represents $(2^j,n-1,n-1)$, then $L\dotplus {\cal L}_2$ CFI--represents $(2^{j+1},n,n)$. For $n=7$, $2^{n-1}=2^6=64$, and we have:

$(2,7,7)$ is CFI--represented by ${\cal L}_3\boxplus {\cal L}_3\boxplus {\cal L}_3\boxplus {\cal L}_3\boxplus {\cal L}_3$;

$(4,7,7)$ is CFI--represented by $({\cal L}_3\boxplus {\cal L}_3\boxplus {\cal L}_3\boxplus {\cal L}_3)\dotplus {\cal L}_2$;

$(8,7,7)$ is CFI--represented by $M_3\dotplus {\cal L}_3$;

$(16,7,7)$ is CFI--represented by ${\cal L}_2^2\dotplus {\cal L}_2^2$;

$(32,7,7)$ is CFI--represented by ${\cal L}_2^2\dotplus {\cal L}_4$;

$(64,7,7)$ is CFI--represented by ${\cal L}_7$.

Now an easy induction argument proves (\ref{listcfirepr1}).

\noindent (\ref{listcfirepr2}) For any $n\in \N $ with $n\geq 7$, $(2,n,n)$, $(3,n,n)$ and $(4,n,n)$ are CFI--representable, and, if $(k,n,n)$ is CFI--representable for some $k\in \N ^*$, then $(k+3,n+3,n+3)$ is CFI--representable, thus, if $(k,n,n)$ is CFI--representable for any $k\in \overline{2,n}$, then $(k,n+3,n+3)$ is CFI--representable for any $k\in \overline{2,n+3}$. Thus it suffices to prove that, for any $n\in \{7,9,11\}$ and any $k\in \overline{2,n+1}$, $(k,n,n)$ is CFI--representable; then (\ref{listcfirepr2}) follows by induction. Actually, by the above and the fact that, furthermore, for any $n\in \N $ with $n\geq 7$, $(5,n,n)$ and $(6,n,n)$ are CFI--representable, as well, it remains to prove  that, for any $n\in \{7,9,11\}$ and any $k\in \overline{7,n+1}$, $(k,n,n)$ is CFI--representable.

Since ${\cal L}_2^2$ CFI--represents $(4,4,4)$, it follows that ${\cal L}_3\boxplus ({\cal L}_2\dotplus {\cal L}_2^2\dotplus {\cal L}_2)$ CFI--represents $(4+3,4+3,4+3)=(7,7,7)$.

Since $M_3\dotplus {\cal L}_2$ CFI--represents $(4,6,6)$, it follows that ${\cal L}_3\boxplus ({\cal L}_2\dotplus M_3\dotplus {\cal L}_2\dotplus {\cal L}_2)={\cal L}_3\boxplus ({\cal L}_2\dotplus M_3\dotplus {\cal L}_3)$ CFI--represents $(4+3,6+3,6+3)=(7,9,9)$.

Since $\boxplus _{i=1}^4{\cal L}_3$ CFI--represents $(2,6,6)$ and ${\cal L}_2^2$ CFI--represents $(4,4,4)$, it follows that $(\boxplus _{i=1}^4{\cal L}_3)\dotplus {\cal L}_2^2$ CFI--represents $(2\cdot 4,6+4-1,6+4-1)=(8,9,9)$.

And $(8,7,7)$ is CFI--represented by $M_3\dotplus {\cal L}_3$.

Since ${\cal L}_2^2\dotplus {\cal L}_2$ CFI--represents $(8,5,5)$, it follows that ${\cal L}_3\boxplus ({\cal L}_2\dotplus {\cal L}_2^2\dotplus {\cal L}_2\dotplus {\cal L}_2)={\cal L}_3\boxplus ({\cal L}_2\dotplus {\cal L}_2^2\dotplus {\cal L}_3)$ CFI--represents $(8+1,5+4,5+4)=(9,9,9)$.

$(7,6,6)$ is CFI--represented by ${\cal L}_4\boxplus {\cal L}_4$ (as well as ${\cal L}_3\boxplus {\cal L}_5$), hence, for instance, ${\cal L}_3\boxplus ({\cal L}_2\dotplus ({\cal L}_4\boxplus {\cal L}_4)\dotplus {\cal L}_2)$ CFI--represents $(7+3,6+3,6+3)=(10,9,9)$.

Since $(6,7,7)$, $(7,7,7)$ and $(8,7,7)$ are CFI--representable by the above and (\ref{listcfirepr1}), it follows that $(6+1,7+4,7+4)=(7,11,11)$, $(7+1,7+4,7+4)=(8,11,11)$ and $(8+1,7+4,7+4)=(9,11,11)$ are CFI--representable.

Since $\boxplus _{i=1}^5{\cal L}_3$ CFI--represents $(2,7,7)$ and $N_5$ CFI--represents $(5,5,5)$, it follows that $(\boxplus _{i=1}^5{\cal L}_3)\dotplus N_5$ CFI--represents $(2\cdot 5,7+5-1,7+5-1)=(10,11,11)$.

Since $N_5$ CFI--represents $(5,5,5)$, it follows that ${\cal L}_3\boxplus ({\cal L}_2\dotplus ({\cal L}_3\boxplus ({\cal L}_2\dotplus N_5\dotplus {\cal L}_2))\dotplus {\cal L}_2)$ CFI--represents $(5+3+3,5+3+3,5+3+3)=(11,11,11)$.

${\cal L}_4\boxplus {\cal L}_3\boxplus {\cal L}_3$, $M_3$ and ${\cal L}_2$ CFI--represent $(3,6,6)$, $(2,5,5)$ and $(2,2,2)$, respectively, hence $({\cal L}_4\boxplus {\cal L}_3\boxplus {\cal L}_3)\dotplus M_3\dotplus {\cal L}_2$ CFI--represents $(3\cdot 2\cdot 2,6+5-3+1,6+5-3+1)=(12,11,11)$.\end{proof}

\begin{remark} Many results can be derived from Proposition \ref{listcfirepr}. For instance, using ordinal sums, which, of course, we can iterate, to obtain more results, we get that, for any $n,m,l\in \N ^*$ such that $(l,m,m)$ is CFI--representable and $n\geq 7$:\begin{enumerate}
\item for any $j\in \overline{1,n-1}$, $(2^{j}\cdot l,n+m-1,n+m-1)$ is CFI--representable;
\item if $n\neq 8$, then, for any $k\in \overline{2,n+1}$, $(k\cdot l,n+m-1,n+m-1)$ is CFI--representable.
\end{enumerate}

Thus, for instance, if $n\in \N $ is such that $n\geq 7$ and $n\neq 8$, then, for any $k\in \overline{2,n+1}$, $(k^2,2n-1,2n-1)$ is CFI--representable and, more generally, $(k^s,sn-s+1,sn-s+1)$ is CFI--representable for any $s\in \N ^*$.\label{mult2}\end{remark}

\section{On the Largest Numbers of Congruences of Finite Lattices}
\label{thelarge}

Let $n\in \N ^*$ and $L$ be a lattice with $|L|=n$.

\begin{remark} Since $L$ is finite, its meet--irreducibles are strictly meet--irreducible, its join--irreducibles are strictly join--irreducible, and, for any $u,v\in L$ with $u<v$, $[u,v]_L$ contains at least one successor of $u$ and one predecessor of $v$. So, for any $a,b\in L$, $[a,b]_L$ is a narrows iff $a\prec b$, $b$ is the unique successor of $a$ and $a$ is the unique predecessor of $b$ in $L$.\label{lfin}\end{remark}

\begin{lemma}{\rm \cite{gcze}} If $L$ is non--trivial, then:\begin{enumerate}
\item\label{lgcze1} $\emptyset \neq At({\rm Con}(L))\subseteq \{{\rm con}(a,b)\ |\ a,b\in L,a\prec b\}$;
\item\label{lgcze2} for any $\theta \in At({\rm Con}(L))$, $|{\rm Con}(L/\theta )|\geq |{\rm Con}(L)|/2$;
\item\label{lgcze3} for any $a,b\in L$ such that $a\prec b$: $[a,b]_L$ is a narrows iff $L/{\rm con}(a,b)=\{\{a,b\}\}\cup \{\{x\}\ |\ x\in L\setminus \{a,b\}\}$ iff $|L/{\rm con}(a,b)|=|L|-1$;
\item\label{lgcze4} for any $a,b\in L$ such that $a\prec b$ and $|L/{\rm con}(a,b)|=|L|-2$, we have one of the following situations:\begin{itemize}
\item $a$ is meet--reducible, case in which $a\prec c$ for some $c\in L\setminus \{b\}$ such that $b\prec b\vee c$, $c\prec b\vee c$ and $L/{\rm con}(a,b)=\{\{a,b\},\{c,b\vee c\}\}\cup \{\{x\}\ |\ x\in L\setminus \{a,b,c,b\vee c\}\}$;
\item $b$ is join--reducible, case in which, dually, $c\prec b$ for some $c\in L\setminus \{a\}$ such that $a\wedge c\prec a$, $a\wedge c\prec c$ and $L/{\rm con}(a,b)=\{\{b\wedge c,c\},\{a,b\}\}\cup \{\{x\}\ |\ x\in L\setminus \{b\wedge c,c,a,b\}\}$.\end{itemize}\end{enumerate}\label{lgcze}\end{lemma}

\begin{remark} Let $a,b\in L$ with $a\neq b$. Also, let $\theta \in {\rm Con}(L)$. If $a\prec b$ and $a/\theta \neq b/\theta $, then, clearly, $a/\theta \prec b/\theta $. If $a/\theta \prec b/\theta $, then there exists no $u\in [a,b]_L\setminus (a/\theta \cup b/\theta )$, because otherwise we would have $a/\theta <u/\theta <b/\theta $. Let us also note that $a/\theta \leq b/\theta $ iff $a\vee b\in b/\theta $ iff $a\wedge b\in a/\theta $ iff $a\leq x$ for some $x\in b/\theta $ iff $w\leq b$ for some $w\in a/\theta $.

By Lemma \ref{lgcze}, (\ref{lgcze3}), if $[a,b]_L$ is a narrows, then ${\rm con}(a,b)$ collapses a single pair of elements, thus, clearly, ${\rm con}(a,b)\in At({\rm Con}(L))$. Since $a/{\rm con}(a,b)=b/{\rm con}(a,b)$, we have $|L/{\rm con}(a,b)|\leq |L|-1$, hence the second equivalence in Lemma \ref{lgcze}, (\ref{lgcze3}), is clear.

By Lemma \ref{lgcze}, (\ref{lgcze3}), if $|L/{\rm con}(a,b)|<|L|-1$, as in Lemma \ref{lgcze}, (\ref{lgcze4}), then $[a,b]_L$ is not a narrows, hence $a$ is meet--reducible, so that $a$ has a successor different from $b$, or $b$ is join--reducible, so that $b$ has a predecessor different from $a$, by Remark \ref{lfin}. With the notations in Lemma \ref{lgcze}, (\ref{lgcze4}), if $|L|-|L/{\rm con}(a,b)|=2$ and, for instance, $a$ is meet--reducible, then, simply, the fact that $(a,b),(c,b\vee c)=(a\vee c,b\vee c)\in {\rm con}(a,b)$ implies that $L/{\rm con}(a,b)=\{\{a,b\},\{c,b\vee c\}\}\cup \{\{x\}\ |\ x\in L\setminus \{a,b,c,b\vee c\}\}$, with $a/{\rm con}(a,b)\neq x/{\rm con}(a,b)\neq c/{\rm con}(a,b)$ for all $x\in L\setminus \{a,b,c,b\vee c\}$ and $a/{\rm con}(a,b)\neq c/{\rm con}(a,b)$, which, along with the fact that $a\prec c$, as above, proves that $a/{\rm con}(a,b)\prec c/{\rm con}(a,b)=(b\vee c)/{\rm con}(a,b)$.\label{onlgcze}\end{remark}

\begin{remark} Remark \ref{nleq4} and the fact that $2^{1-1}=1=2^{2-2}$, $2^{2-1}=2=2^{3-2}$ and $2^{3-1}=4=2^{4-2}$ give us:\begin{itemize}
\item if $|{\rm Con}(L)|<2^{n-1}$, then $n\geq 4$;
\item if $|{\rm Con}(L)|<2^{n-2}$, then $n\geq 5$.\end{itemize}\label{ngeq}\end{remark}

\begin{theorem}\begin{enumerate}
\item{\rm \cite{free,gcze}}\label{gczetgh1} $|{\rm Con}(L)|\leq 2^{n-1}$ and: $|{\rm Con}(L)|=2^{n-1}$ iff $L\cong {\cal L}_n$.
\item{\rm \cite{gcze}}\label{gczetgh2} if $|{\rm Con}(L)|<2^{n-1}$, then $|{\rm Con}(L)|\leq 2^{n-2}$ and: $|{\rm Con}(L)|=2^{n-2}$ iff $L\cong {\cal L}_k\dotplus {\cal L}_2^2\dotplus {\cal L}_{n-k-2}$ for some $k\in \overline{1,n-3}$.\end{enumerate}\label{gczetgh}\end{theorem}

Following the line of the proof from \cite{gcze} of Theorem \ref{gczetgh}, now we prove:

\begin{theorem} If $|{\rm Con}(L)|<2^{n-2}$, then $n\geq 5$ and:\begin{enumerate}
\item\label{nextcgno1} $|{\rm Con}(L)|\leq 5\cdot 2^{n-5}=2^{n-3}+2^{n-5}$ and: $|{\rm Con}(L)|=5\cdot 2^{n-5}$ iff $L\cong {\cal L}_k\dotplus N_5\dotplus {\cal L}_{n-k-3}$ for some $k\in \overline{1,n-4}$;
\item\label{nextcgno2} if $|{\rm Con}(L)|<5\cdot 2^{n-5}$, then $|{\rm Con}(L)|\leq 2^{n-3}$ and: $|{\rm Con}(L)|=2^{n-3}$ iff either $n\geq 6$ and $L\cong {\cal L}_k\dotplus ({\cal L}_2\times {\cal L}_3)\dotplus {\cal L}_{n-k-4}$ for some $k\in \overline{1,n-5}$, or $n\geq 7$ and $L\cong {\cal L}_k\dotplus {\cal L}_2^2\dotplus {\cal L}_m\dotplus {\cal L}_2^2\dotplus {\cal L}_{n-k-m-4}$ for some $k,m\in \N ^*$ such that $k+m\leq n-5$;
\item\label{nextcgno3} if $|{\rm Con}(L)|<2^{n-3}$, then $|{\rm Con}(L)|\leq 7\cdot 2^{n-6}=2^{n-4}+2^{n-5}+2^{n-6}$.\end{enumerate}\label{nextcgno}\end{theorem}

\begin{proof} Assume that $|{\rm Con}(L)|<2^{n-2}<2^{n-1}$, so that $n\geq 5$ by Remark \ref{ngeq}. We shall prove the statements in the enunciation by induction on $n\in \N $, $n\geq 5$. We shall identify the lattices up to isomorphism.

The five--element lattices are: $M_3$, $N_5$, ${\cal L}_2\dotplus {\cal L}_2^2$, ${\cal L}_2^2\dotplus {\cal L}_2$ and ${\cal L}_5$, whose numbers of congruences are: $2$, $5$, $8$, $8$ and $2^4=16$, respectively. The five--element lattices with strictly less than $2^{5-2}=8$ congruences are $M_3$ and $N_5$, out of which $N_5\cong {\cal L}_1\dotplus N_5\dotplus {\cal L}_{5-1-3}$, is of the form in (\ref{nextcgno1}) and has $5=5\cdot 2^{5-5}$ congruences, while $M_3$ has $2<4=2^{5-3}$ congruences. From this fact and Remark \ref{ngeq}, it follows that, if $|{\rm Con}(L)|=2^{n-3}$, then $n\geq 6$.

The six--element lattices are: ${\cal L}_3\boxplus {\cal L}_3\boxplus {\cal L}_3\boxplus {\cal L}_3$, ${\cal L}_3\boxplus {\cal L}_3\boxplus {\cal L}_4$, $M_3\dotplus {\cal L}_2$, ${\cal L}_2\dotplus M_3$, ${\cal L}_3\boxplus ({\cal L}_2^2\dotplus {\cal L}_2)$, ${\cal L}_3\boxplus ({\cal L}_2\dotplus {\cal L}_2^2)$, ${\cal L}_3\boxplus {\cal L}_5$, ${\cal L}_4\boxplus {\cal L}_4$, ${\cal L}_2\times {\cal L}_3$, $N_5\dotplus {\cal L}_2$, ${\cal L}_2\dotplus N_5$, ${\cal L}_2^2\dotplus {\cal L}_3$, ${\cal L}_3\dotplus {\cal L}_2^2$, ${\cal L}_2\dotplus {\cal L}_2^2\dotplus {\cal L}_2$ and ${\cal L}_6$, whose numbers of congruences are: $2$, $3$, $4$, $4$, $6$, $6$, $7$, $7$, $8$, $10$, $10$, $16$, $16$, $16=2^{6-2}$ and $32=2^{6-1}$, respectively. So, the third largest number of congruences of a six--element lattice is $10=5\cdot 2^{6-5}$, the fourth largest is $8=2^{6-3}$ and the fifth largest is $7=7\cdot 2^{6-6}$. As above, we notice that $N_5\dotplus {\cal L}_2$ and ${\cal L}_2\dotplus N_5$ are of the form in (\ref{nextcgno1}) and ${\cal L}_2\times {\cal L}_3$ is of the first form in (\ref{nextcgno2}).

It is easy to construct, as above, the $7$--element lattices, and see that the ones with strictly less than $2^{7-2}=32$ congruences are:\begin{itemize}
\item the ones having $20=5\cdot 2^{7-5}$ congruences, namely $N_5\dotplus {\cal L}_3$, ${\cal L}_3\dotplus N_5$ and ${\cal L}_2\dotplus N_5\dotplus {\cal L}_2$, all of the form in (\ref{nextcgno1});
\item the ones having $16=2^{7-3}$ congruences, namely $({\cal L}_2\times {\cal L}_3)\dotplus {\cal L}_2$ and ${\cal L}_2\dotplus ({\cal L}_2\times {\cal L}_3)$, which are of the first form in (\ref{nextcgno2}), as well as ${\cal L}_2^2\dotplus {\cal L}_2^2$, which is of the second form in (\ref{nextcgno2});
\item the ones having $14=7\cdot 2^{7-6}$ congruences, namely $({\cal L}_3\boxplus {\cal L}_5)\dotplus {\cal L}_2$, ${\cal L}_2\dotplus ({\cal L}_3\boxplus {\cal L}_5)$, $({\cal L}_4\boxplus {\cal L}_4)\dotplus {\cal L}_2$ and ${\cal L}_2\dotplus ({\cal L}_4\boxplus {\cal L}_4)$;
\item and the ones having strictly less than $14$ congruences.\end{itemize}

Now assume that $n\geq 8$ and lattices of cardinality at most $n-1$ fulfill the statements in the enunciation.

Note that, in the rest of this proof, whenever $|{\rm Con}(L)|=5\cdot 2^{n-5}$, $L$ is of the form in (\ref{nextcgno1}), and, whenever $|{\rm Con}(L)|=2^{n-3}$, $L$ is of one of the forms in (\ref{nextcgno2}).

Let $\theta \in At({\rm Con}(L))$. By Lemma \ref{lgcze}, (\ref{lgcze1}), at least one such $\theta $ exists, and $\theta ={\rm con}(a,b)$ for some $a,b\in L$ with $a\prec b$. By Lemma \ref{lgcze}, (\ref{lgcze3}), and Theorem \ref{gczetgh}, (\ref{gczetgh1}), $|L/\theta |\leq n-1$, thus $|{\rm Con}(L/\theta )|\leq 2^{n-2}$.

{\bf Case 1:} Assume that $|L/\theta |=n-1$, so that, according to Lemma \ref{lgcze}, (\ref{lgcze2}), $L/\theta =\{\{a,b\}\}\cup \{\{x\}\ |\ x\in L\setminus \{a,b\}\}$ and $[a,b]_L$ is a narrows, thus $b$ is the unique successor of $a$ and $a$ is the unique predecessor of $b$.

Then, clearly, $L/\theta \cong {\cal L}_{n-1}$ iff $L\cong {\cal L}_n$, hence, by Theorem \ref{gczetgh}, (\ref{gczetgh1}), $|{\rm Con}(L/\theta )|=2^{n-2}$ iff $|{\rm Con}(L)|=2^{n-1}$, which would contradict the hypothesis that $|{\rm Con}(L)|<2^{n-2}$ of the present theorem. Thus $|{\rm Con}(L/\theta )|<2^{n-2}$, so that $|{\rm Con}(L/\theta )|\leq 2^{n-3}$ according to Theorem \ref{gczetgh}, (\ref{gczetgh2}).

{\em Subcase 1.1:} Assume that $|{\rm Con}(L/\theta )|=2^{n-3}$, which, according to Theorem \ref{gczetgh}, (\ref{gczetgh2}), means that $L/\theta \cong {\cal L}_k\dotplus {\cal L}_2^2\dotplus {\cal L}_{n-k-3}\cong {\cal L}_k\dotplus ({\cal L}_3\boxplus {\cal L}_3)\dotplus {\cal L}_{n-k-3}$ for some $k\in \overline{1,n-4}$. If we denote the elements of $L/\theta $ as in the leftmost diagram below, with $x,y,z,u\in L$, and we also consider the facts that $|L|-|L/\theta |=1$, $a$ has the unique successor $b$ and $b$ has the unique predecessor $a$, $a/\theta =b/\theta =\{a,b\}$ and $v/\theta =\{v\}$ for all $v\in L\setminus \{a,b\}$, then we notice that $L$ is in one of the following situations, represented in the three diagrams to the right of that of $L/\theta $:\begin{itemize}
\item if $a/\theta =b/\theta \leq x/\theta $, then $b\leq x$ and $L\cong {\cal L}_2\dotplus L/\theta \cong {\cal L}_{k+1}\dotplus {\cal L}_2^2\dotplus {\cal L}_{n-k-3}$, while, if $a/\theta =b/\theta \geq u/\theta $, then $a\geq u$ and $L\cong L/\theta \dotplus {\cal L}_2\cong {\cal L}_k\dotplus {\cal L}_2^2\dotplus {\cal L}_{n-k-2}$, but in these situations $|{\rm Con}(L)|=2^{n-2}$, which contradicts the hypothesis that $|{\rm Con}(L)|<2^{n-2}$ of the theorem;
\item if $x/\theta <a/\theta =b/\theta <u/\theta $, then $x<a<b<u$, hence $\{a,b\}\cap \{y,z\}\neq \emptyset $, so that $L\cong {\cal L}_k\dotplus ({\cal L}_3\boxplus {\cal L}_4)\dotplus {\cal L}_{n-k-3}\cong {\cal L}_k\dotplus N_5\dotplus {\cal L}_{n-k-3}$, thus $|{\rm Con}(L)|=2^{k-1}\cdot 5\cdot 2^{n-k-4}=5\cdot 2^{n-5}$.\end{itemize}

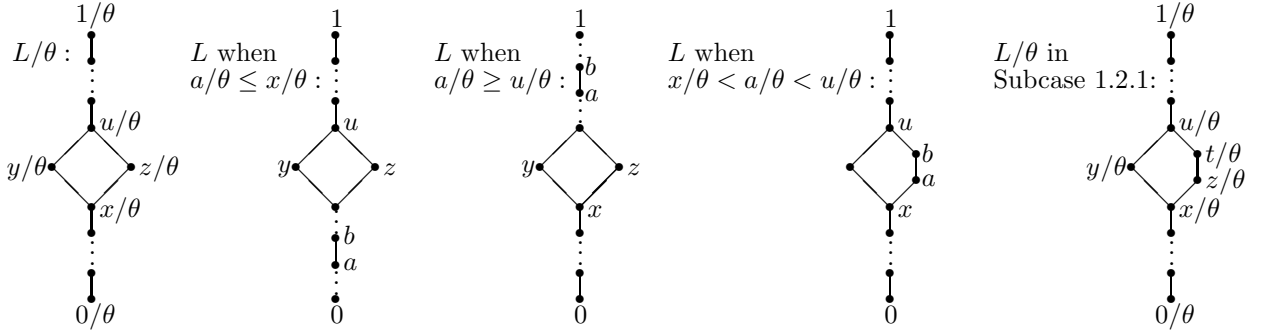
\begin{center}\begin{tabular}{ccccc}
\begin{picture}(40,110)(0,0)
\put(-10,90){$L/\theta :$}
\put(20,0){\circle*{3}}
\put(20,10){\circle*{3}}
\put(19,13){$\vdots $}
\put(20,25){\circle*{3}}
\put(20,35){\circle*{3}}
\put(20,0){\line(0,1){10}}
\put(20,35){\line(-1,1){15}}
\put(20,35){\line(1,1){15}}
\put(20,35){\line(0,-1){10}}
\put(5,50){\circle*{3}}
\put(35,50){\circle*{3}}
\put(20,65){\circle*{3}}
\put(20,75){\circle*{3}}
\put(19,78){$\vdots $}
\put(20,90){\circle*{3}}
\put(20,100){\circle*{3}}
\put(23,31){$x/\theta $}
\put(-12,47){$y/\theta $}
\put(38,47){$z/\theta $}
\put(23,64){$u/\theta $}
\put(14,-9){$0/\theta $}
\put(14,105){$1/\theta $}
\put(20,65){\line(0,1){10}}
\put(20,65){\line(-1,-1){15}}
\put(20,65){\line(1,-1){15}}
\put(20,100){\line(0,-1){10}}
\end{picture}
&\hspace*{37pt}
\begin{picture}(40,110)(0,0)
\put(-35,90){$L$ when}
\put(-35,80){$a/\theta \leq x/\theta :$}
\put(20,0){\circle*{3}}
\put(20,13){\circle*{3}}
\put(20,13){\line(0,1){10}}
\put(20,35){\line(-1,1){15}}
\put(20,35){\line(1,1){15}}
\put(19,1.7){$\vdots $}
\put(19,24.7){$\vdots $}
\put(20,23){\circle*{3}}
\put(20,35){\circle*{3}}
\put(5,50){\circle*{3}}
\put(35,50){\circle*{3}}
\put(20,65){\circle*{3}}
\put(20,75){\circle*{3}}
\put(19,78){$\vdots $}
\put(20,90){\circle*{3}}
\put(20,100){\circle*{3}}
\put(23,11){$a$}
\put(23,21){$b$}
\put(-2,48){$y$}
\put(38,47){$z$}
\put(23,64){$u$}
\put(18,-9){$0$}
\put(18,103){$1$}
\put(20,65){\line(0,1){10}}
\put(20,65){\line(-1,-1){15}}
\put(20,65){\line(1,-1){15}}
\put(20,100){\line(0,-1){10}}
\end{picture}
&\hspace*{37pt}
\begin{picture}(40,110)(0,0)
\put(-35,90){$L$ when}
\put(-35,80){$a/\theta \geq u/\theta :$}
\put(20,0){\circle*{3}}
\put(20,10){\circle*{3}}
\put(19,13){$\vdots $}
\put(20,25){\circle*{3}}
\put(20,35){\circle*{3}}
\put(20,0){\line(0,1){10}}
\put(20,35){\line(1,1){15}}
\put(20,35){\line(0,-1){10}}
\put(5,50){\circle*{3}}
\put(35,50){\circle*{3}}
\put(20,65){\circle*{3}}
\put(20,78){\circle*{3}}
\put(19,66.7){$\vdots $}
\put(19,87.7){$\vdots $}
\put(20,88){\circle*{3}}
\put(20,100){\circle*{3}}
\put(23,31){$x$}
\put(-2,48){$y$}
\put(38,47){$z$}
\put(18,-9){$0$}
\put(18,103){$1$}
\put(22,75){$a$}
\put(22,85){$b$}
\put(20,35){\line(-1,1){15}}
\put(20,35){\line(1,1){15}}
\put(20,88){\line(0,-1){10}}
\put(20,65){\line(-1,-1){15}}
\put(20,65){\line(1,-1){15}}
\end{picture}
&\hspace*{62pt}
\begin{picture}(40,110)(0,0)
\put(-64,90){$L$ when}
\put(-64,80){$x/\theta <a/\theta <u/\theta :$}
\put(20,0){\circle*{3}}
\put(20,10){\circle*{3}}
\put(19,13){$\vdots $}
\put(20,25){\circle*{3}}
\put(20,35){\circle*{3}}
\put(20,0){\line(0,1){10}}
\put(20,35){\line(-1,1){15}}
\put(20,35){\line(1,1){10}}
\put(20,35){\line(0,-1){10}}
\put(5,50){\circle*{3}}
\put(20,65){\circle*{3}}
\put(20,75){\circle*{3}}
\put(19,78){$\vdots $}
\put(20,90){\circle*{3}}
\put(20,100){\circle*{3}}
\put(23,31){$x$}
\put(23,64){$u$}
\put(18,-9){$0$}
\put(18,103){$1$}
\put(20,65){\line(0,1){10}}
\put(20,65){\line(-1,-1){15}}
\put(20,65){\line(1,-1){10}}
\put(30,55){\line(0,-1){10}}
\put(30,55){\circle*{3}}
\put(30,45){\circle*{3}}
\put(32,43){$a$}
\put(32,52){$b$}
\put(20,100){\line(0,-1){10}}
\end{picture}
&\hspace*{51pt}
\begin{picture}(40,110)(0,0)
\put(-47,90){$L/\theta $ in}
\put(-47,80){Subcase 1.2.1:}
\put(20,0){\circle*{3}}
\put(20,10){\circle*{3}}
\put(19,13){$\vdots $}
\put(20,25){\circle*{3}}
\put(20,35){\circle*{3}}
\put(20,0){\line(0,1){10}}
\put(20,35){\line(-1,1){15}}
\put(20,35){\line(1,1){10}}
\put(20,35){\line(0,-1){10}}
\put(5,50){\circle*{3}}
\put(20,65){\circle*{3}}
\put(20,75){\circle*{3}}
\put(19,78){$\vdots $}
\put(20,90){\circle*{3}}
\put(20,100){\circle*{3}}
\put(23,31){$x/\theta $}
\put(-12,47){$y/\theta $}
\put(33,42){$z/\theta $}
\put(33,52){$t/\theta $}
\put(23,64){$u/\theta $}
\put(14,-9){$0/\theta $}
\put(14,105){$1/\theta $}
\put(20,65){\line(0,1){10}}
\put(20,65){\line(-1,-1){15}}
\put(20,65){\line(1,-1){10}}
\put(30,55){\line(0,-1){10}}
\put(30,55){\circle*{3}}
\put(30,45){\circle*{3}}
\put(20,100){\line(0,-1){10}}
\end{picture}\end{tabular}\end{center}\vspace*{3pt}

{\em Subcase 1.2:} Assume that $|{\rm Con}(L/\theta )|<2^{n-3}$, so that, by the induction hypothesis, $|{\rm Con}(L/\theta )|\leq 5\cdot 2^{n-6}$, thus $|{\rm Con}(L)|\leq 2\cdot 5\cdot 2^{n-6}=5\cdot 2^{n-5}$ by Lemma \ref{lgcze}, (\ref{lgcze2}).

{\em Subcase 1.2.1:} Assume that $|{\rm Con}(L/\theta )|=5\cdot 2^{n-6}$, which, by the induction hypothesis, means that $L/\theta \cong {\cal L}_k\dotplus N_5\dotplus {\cal L}_{n-k-4}$ for some $k\in \overline{1,n-5}$, so that $L$ is in one of the following situations, that we separate as in Subcase 1.1, where the elements of $L/\theta $ are denoted as in the rightmost diagram above, with $x,y,z,t,u\in L$:\begin{itemize}
\item if $a/\theta =b/\theta \leq x/\theta $, then $a<b\leq x$ and $L\cong {\cal L}_2\dotplus L/\theta \cong {\cal L}_{k+1}\dotplus N_5\dotplus {\cal L}_{n-k-4}$, while, if $a/\theta =b/\theta \geq u/\theta $, then $u\leq a<b$ and $L\cong L/\theta \dotplus {\cal L}_2\cong {\cal L}_k\dotplus N_5\dotplus {\cal L}_{n-k-3}$, hence $|{\rm Con}(L)|=2\cdot |{\rm Con}(L/\theta )|=5\cdot 2^{n-5}$;
\item if $x/\theta <a/\theta =b/\theta <u/\theta $, then $x<a<b<u$ and: either $\{a,b\}\cap \{z,t\}\neq \emptyset $, case in which $a,b,z,t$ are pairwise comparable, because otherwise $a$ would be meet--reducible or $b$ would be join--reducible, thus $L\cong {\cal L}_k\dotplus ({\cal L}_3\boxplus {\cal L}_5)\dotplus {\cal L}_{n-k-4}$, or $y\in \{a,b\}$, so that $L\cong {\cal L}_k\dotplus ({\cal L}_4\boxplus {\cal L}_4)\dotplus {\cal L}_{n-k-4}$, hence $|{\rm Con}(L)|=2^{k-1}\cdot (2^2+3)\cdot 2^{n-k-5}=7\cdot 2^{n-6}$.\end{itemize}

The following subcases can be treated exactly as above. For brevity, we shall only indicate the shapes of the lattices in the remaining part of the proof.

{\em Subcase 1.2.2:} Assume that $|{\rm Con}(L/\theta )|<5\cdot 2^{n-6}$, so that $|{\rm Con}(L/\theta )|\leq 2^{n-4}$ by the induction hypothesis, hence $|{\rm Con}(L)|\leq 2\cdot 2^{n-4}=2^{n-3}$ by Lemma \ref{lgcze}, (\ref{lgcze2}).

{\em Subcase 1.2.2.1:} Assume that $|{\rm Con}(L/\theta )|=2^{n-4}$, which, by the induction hypothesis, means that either $L/\theta \cong {\cal L}_r\dotplus ({\cal L}_2\times {\cal L}_3)\dotplus {\cal L}_{n-r-5}$ for some $r\in \overline{1,n-6}$, or $L/\theta \cong {\cal L}_k\dotplus {\cal L}_2^2\dotplus {\cal L}_m\dotplus {\cal L}_2^2\dotplus {\cal L}_{n-k-m-5}$ for some $k,m\in \N ^*$ such that $k+m\leq n-6$, so that $L$ is in one of the following situations:\begin{itemize}
\item $L\cong {\cal L}_{r+1}\dotplus ({\cal L}_2\times {\cal L}_3)\dotplus {\cal L}_{n-r-5}$ or $L\cong {\cal L}_r\dotplus ({\cal L}_2\times {\cal L}_3)\dotplus {\cal L}_{n-r-4}$ or $L\cong {\cal L}_{k+1}\dotplus {\cal L}_2^2\dotplus {\cal L}_m\dotplus {\cal L}_2^2\dotplus {\cal L}_{n-k-m-5}$ or $L/\theta \cong {\cal L}_k\dotplus {\cal L}_2^2\dotplus {\cal L}_{m+1}\dotplus {\cal L}_2^2\dotplus {\cal L}_{n-k-m-5}$ or $L/\theta \cong {\cal L}_k\dotplus {\cal L}_2^2\dotplus {\cal L}_m\dotplus {\cal L}_2^2\dotplus {\cal L}_{n-k-m-4}$, so that $|{\rm Con}(L)|=2\cdot |{\rm Con}(L/\theta )|=2^{n-3}$;
\item $L\cong {\cal L}_k\dotplus N_5\dotplus {\cal L}_m\dotplus {\cal L}_2^2\dotplus {\cal L}_{n-k-m-5}$ or $L\cong {\cal L}_k\dotplus {\cal L}_2^2\dotplus {\cal L}_m\dotplus N_5\dotplus {\cal L}_{n-k-m-5}$, case in which $|{\rm Con}(L)|=5\cdot 2^2\cdot 2^{k-1+m-1+n-k-m-6}=5\cdot 2^{n-6}<7\cdot 2^{n-6}$;
\item $L\cong {\cal L}_r\dotplus G\dotplus {\cal L}_{n-r-5}$ or $L\cong {\cal L}_r\dotplus G^{\prime }\dotplus {\cal L}_{n-r-5}$ or $L\cong {\cal L}_r\dotplus H\dotplus {\cal L}_{n-r-5}$ or $L\cong {\cal L}_r\dotplus H^{\prime }\dotplus {\cal L}_{n-r-5}$ or $L\cong {\cal L}_r\dotplus K\dotplus {\cal L}_{n-r-5}$ or $L\cong {\cal L}_r\dotplus K^{\prime }\dotplus {\cal L}_{n-r-5}$, where $G$, $H$ and $K$ are the following gluings of a pentagon with a rhombus and $G^{\prime }$, $H^{\prime }$ and $K^{\prime }$ are the duals of $G$, $H$ and $K$, respectively, so that $|{\rm Con}(L)|=9\cdot 2^{r-1+n-r-6}=9\cdot 2^{n-7}<14\cdot 2^{n-7}=7\cdot 2^{n-6}$ since $|{\rm Con}(G)|=|{\rm Con}(H)|=|{\rm Con}(K)|=9$, which is simple to verify, and thus $|{\rm Con}(G^{\prime })|=|{\rm Con}(H^{\prime })|=|{\rm Con}(K^{\prime })|=9$, as well; below we are indicating the positions of $a$ and $b$ in these copies of $G$, $H$, $K$, $G^{\prime }$, $H^{\prime }$ and $K^{\prime }$ from $L$:\end{itemize}

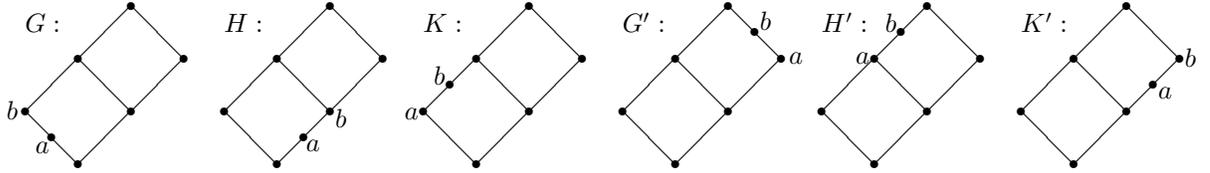
\begin{center}\begin{tabular}{cccccc}
\hspace*{-15pt}
\begin{picture}(40,60)(0,0)
\put(0,50){$G:$}
\put(20,0){\circle*{3}}
\put(0,20){\circle*{3}}
\put(40,20){\circle*{3}}
\put(20,40){\circle*{3}}
\put(60,40){\circle*{3}}
\put(40,60){\circle*{3}}
\put(10,10){\circle*{3}}
\put(4,4){$a$}
\put(-7,17){$b$}
\put(20,0){\line(1,1){40}}
\put(20,0){\line(-1,1){20}}
\put(40,60){\line(-1,-1){40}}
\put(40,60){\line(1,-1){20}}
\put(20,40){\line(1,-1){20}}
\end{picture}
&\hspace*{20pt}
\begin{picture}(40,60)(0,0)
\put(0,50){$H:$}
\put(20,0){\circle*{3}}
\put(0,20){\circle*{3}}
\put(40,20){\circle*{3}}
\put(20,40){\circle*{3}}
\put(60,40){\circle*{3}}
\put(40,60){\circle*{3}}
\put(30,10){\circle*{3}}
\put(31,5){$a$}
\put(42,14){$b$}
\put(20,0){\line(1,1){40}}
\put(20,0){\line(-1,1){20}}
\put(40,60){\line(-1,-1){40}}
\put(40,60){\line(1,-1){20}}
\put(20,40){\line(1,-1){20}}
\end{picture}
&\hspace*{20pt}
\begin{picture}(40,60)(0,0)
\put(0,50){$K:$}
\put(20,0){\circle*{3}}
\put(0,20){\circle*{3}}
\put(40,20){\circle*{3}}
\put(20,40){\circle*{3}}
\put(60,40){\circle*{3}}
\put(40,60){\circle*{3}}
\put(10,30){\circle*{3}}
\put(4,30){$b$}
\put(-7,17){$a$}
\put(20,0){\line(1,1){40}}
\put(20,0){\line(-1,1){20}}
\put(40,60){\line(-1,-1){40}}
\put(40,60){\line(1,-1){20}}
\put(20,40){\line(1,-1){20}}
\end{picture}
&\hspace*{20pt}
\begin{picture}(40,60)(0,0)
\put(0,50){$G^{\prime }:$}
\put(20,0){\circle*{3}}
\put(0,20){\circle*{3}}
\put(40,20){\circle*{3}}
\put(20,40){\circle*{3}}
\put(60,40){\circle*{3}}
\put(40,60){\circle*{3}}
\put(50,50){\circle*{3}}
\put(52,51){$b$}
\put(63,38){$a$}
\put(20,0){\line(1,1){40}}
\put(20,0){\line(-1,1){20}}
\put(40,60){\line(-1,-1){40}}
\put(40,60){\line(1,-1){20}}
\put(20,40){\line(1,-1){20}}
\end{picture}
&\hspace*{20pt}
\begin{picture}(40,60)(0,0)
\put(0,50){$H^{\prime }:$}
\put(20,0){\circle*{3}}
\put(0,20){\circle*{3}}
\put(40,20){\circle*{3}}
\put(20,40){\circle*{3}}
\put(60,40){\circle*{3}}
\put(40,60){\circle*{3}}
\put(30,50){\circle*{3}}
\put(24,50){$b$}
\put(13,38){$a$}
\put(20,0){\line(1,1){40}}
\put(20,0){\line(-1,1){20}}
\put(40,60){\line(-1,-1){40}}
\put(40,60){\line(1,-1){20}}
\put(20,40){\line(1,-1){20}}
\end{picture}
&\hspace*{20pt}
\begin{picture}(40,60)(0,0)
\put(0,50){$K^{\prime }:$}
\put(20,0){\circle*{3}}
\put(0,20){\circle*{3}}
\put(40,20){\circle*{3}}
\put(20,40){\circle*{3}}
\put(60,40){\circle*{3}}
\put(40,60){\circle*{3}}
\put(50,30){\circle*{3}}
\put(52,25){$a$}
\put(62,37){$b$}
\put(20,0){\line(1,1){40}}
\put(20,0){\line(-1,1){20}}
\put(40,60){\line(-1,-1){40}}
\put(40,60){\line(1,-1){20}}
\put(20,40){\line(1,-1){20}}
\end{picture}\end{tabular}\end{center}

{\em Subcase 1.2.2.2:} Assume that $|{\rm Con}(L/\theta )|<2^{n-4}$, so that $|{\rm Con}(L/\theta )|\leq 7\cdot 2^{n-7}$ by the induction hypothesis, thus $|{\rm Con}(L)|\leq 2\cdot 7\cdot 2^{n-7}=7\cdot 2^{n-6}$ by Lemma \ref{lgcze}, (\ref{lgcze2}).

{\bf Case 2:} Assume that $|L/\theta |\leq n-2$, which means that $[a,b]_L$ is not a narrows according to Lemma \ref{lgcze}, (\ref{lgcze3}), hence $a$ is meet--reducible or $b$ is join--reducible. We shall assume that $a$ is meet--reducible, so that $a\prec c$ for some $c\in L\setminus \{b\}$, since $L$ is finite; the case when $b$ is join--reducible shall follow by duality. Since $a\prec b$, $a\prec c$ and $b\neq c$, it follows that $b$ and $c$ are incomparable, hence $b<b\vee c$ and $c<b\vee c$. Since $a/\theta =b/\theta $, we have $(b\vee c)/\theta =(a\vee c)/\theta =c/\theta $.

Let us see that, in this case, $|{\rm Con}(L)|\leq 2^{n-3}$. Indeed, assume by absurdum that $|{\rm Con}(L)|>2^{n-3}$, so that $|{\rm Con}(L/\theta )|>2^{n-4}$ by Lemma \ref{lgcze}, (\ref{lgcze2}). Then Theorem \ref{gczetgh}, (\ref{gczetgh1}) and (\ref{gczetgh2}), ensures us that $|{\rm Con}(L/\theta )|=2^{n-3}$, $|L/\theta |=n-2$ and $L/\theta \cong {\cal L}_{n-2}$; so we are in the situation from Lemma \ref{lgcze}, (\ref{lgcze4}), hence, also using Remark \ref{onlgcze}, we get that $\{a,b\}=a/\theta \prec c/\theta =\{c,b\vee c\}$ and $x/\theta =\{x\}$ for all $x\in L\setminus \{a,b,c,b\vee c\}$, and, since $L/\theta $ is a chain, for all $x,y\in L$, we have either $x/\theta \leq a/\theta \prec c/\theta $ or $a/\theta \prec c/\theta =(b\vee c)/\theta \leq x/\theta $, and either $x/\theta \leq y/\theta $ or $y/\theta \leq x/\theta $, therefore $L\cong {\cal L}_k\dotplus {\cal L}_2^2\dotplus {\cal L}_{n-k-2}$ for some $k\in \overline{1,n-3}$, with $\{a,b,c,b\vee c\}$ being the sublattice of $L$ isomorphic to ${\cal L}_2^2$; but then $|{\rm Con}(L)|=2^{n-2}$, which contradicts the hypothesis that $|{\rm Con}(L)|<2^{n-2}$ of the theorem. Therefore, indeed, $|{\rm Con}(L)|\leq 2^{n-3}$.

{\em Subcase 2.1:} Assume that $|{\rm Con}(L)|=2^{n-3}$, so that $|{\rm Con}(L/\theta )|\geq 2^{n-4}$ by Lemma \ref{lgcze}, (\ref{lgcze2}). This, along with the fact that $|L/\theta |\leq n-2$, and Theorem \ref{gczetgh}, (\ref{gczetgh1}) and (\ref{gczetgh2}), shows that there are only three possibilities:\begin{itemize}
\item {\bf (a)} $|L/\theta |=n-2$ and $|{\rm Con}(L/\theta )|=2^{n-3}$, case in which $L/\theta \cong {\cal L}_{n-2}$, but then, as above, it follows that $L\cong {\cal L}_r\dotplus {\cal L}_2^2\dotplus {\cal L}_{n-r-2}$ for some $r\in \overline{1,n-3}$, so that $|{\rm Con}(L)|=2^{n-2}$, and this contradicts the hypothesis that $|{\rm Con}(L)|<2^{n-2}$ of the theorem;
\item {\bf (b)} $|L/\theta |=n-2$ and $|{\rm Con}(L/\theta )|=2^{n-4}$, case in which $L/\theta \cong {\cal L}_r\dotplus {\cal L}_2^2\dotplus {\cal L}_{n-r-4}$ for some $r\in \overline{1,n-5}$, according to Theorem \ref{gczetgh}, (\ref{gczetgh2});
\item {\bf (c)} $|L/\theta |=n-3$ and $|{\rm Con}(L/\theta )|=2^{n-4}$, case in which $L/\theta \cong {\cal L}_{n-3}$, according to Theorem \ref{gczetgh}, (\ref{gczetgh1}).
\end{itemize}

In the case {\bf (b)}, we have $|L|-|L/\theta |=2$, so that we are in the situation from Lemma \ref{lgcze}, (\ref{lgcze4}), therefore, by Remark \ref{onlgcze}, $\{a,b\}=a/\theta \prec c/\theta =\{c,b\vee c\}$ and, for all $x\in L\setminus (a/\theta \cup c/\theta )=L\setminus \{a,b,c,b\vee c\}$, $x/\theta =\{x\}$ and $x\notin [a,b\vee c]_L$, hence $L$ has one of the following forms:\begin{itemize}
\item $L\cong {\cal L}_s\dotplus {\cal L}_2^2\dotplus {\cal L}_t\dotplus {\cal L}_2^2\dotplus {\cal L}_{n-s-t-4}$ for some $s,t\in \N ^*$ such that $s+t\leq n-5$; in this case, one of the two copies of ${\cal L}_2^2$ from $L$ is $\{a,b,c,b\vee c\}$, $r\in \{s,s+t+2\}$, and, indeed, $|{\rm Con}(L)|=2^{n-3}$;
\item $L\cong {\cal L}_r\dotplus ({\cal L}_2\times {\cal L}_3)\dotplus {\cal L}_{n-r-4}$, case in which, indeed, $|{\rm Con}(L)|=2^{n-3}$, and $a$ and $b\vee c$ belong to the copy of ${\cal L}_2\times {\cal L}_3$ from $L$, in which they are situated as in one of the following two leftmost diagrams;
\item $L\cong {\cal L}_r\dotplus ({\cal L}_3\boxplus ({\cal L}_2^2\dotplus {\cal L}_2))\dotplus {\cal L}_{n-r-4}$ or $L\cong {\cal L}_r\dotplus ({\cal L}_3\boxplus ({\cal L}_2\dotplus {\cal L}_2^2))\dotplus {\cal L}_{n-r-4}$, in which $a$, $b$, $c$ and $b\vee c$ would be positioned in the copy of ${\cal L}_3\boxplus ({\cal L}_2^2\dotplus {\cal L}_2)$, respectively ${\cal L}_3\boxplus ({\cal L}_2\dotplus {\cal L}_2^2)$, as in the following two rightmost diagrams, but then, by Remark \ref{0red}, $|{\rm Con}(L)|=(2+2^2)\cdot 2^{r-1+n-r-5}=6\cdot 2^{n-6}=3\cdot 2^{n-5}<4\cdot 2^{n-5}=2^{n-3}$, which contradicts the hypothesis that $|{\rm Con}(L)|=2^{n-3}$ of Subcase 2.1.\end{itemize}\vspace*{1pt}

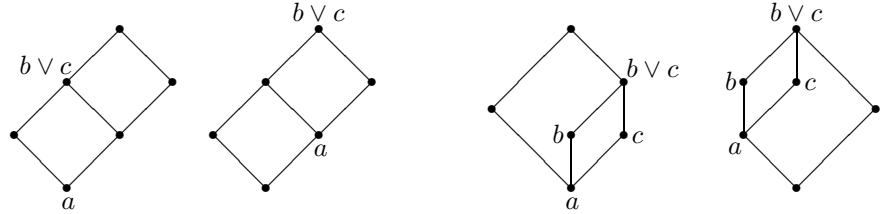
\begin{center}\begin{tabular}{cccc}
\begin{picture}(40,60)(0,0)
\put(20,0){\circle*{3}}
\put(0,20){\circle*{3}}
\put(40,20){\circle*{3}}
\put(20,40){\circle*{3}}
\put(60,40){\circle*{3}}
\put(40,60){\circle*{3}}
\put(18,-8){$a$}
\put(2,43){$b\vee c$}
\put(20,0){\line(1,1){40}}
\put(20,0){\line(-1,1){20}}
\put(40,60){\line(-1,-1){40}}
\put(40,60){\line(1,-1){20}}
\put(20,40){\line(1,-1){20}}
\end{picture}
&\hspace*{20pt}
\begin{picture}(40,60)(0,0)
\put(20,0){\circle*{3}}
\put(0,20){\circle*{3}}
\put(40,20){\circle*{3}}
\put(20,40){\circle*{3}}
\put(60,40){\circle*{3}}
\put(40,60){\circle*{3}}
\put(38,12){$a$}
\put(30,63){$b\vee c$}
\put(20,0){\line(1,1){40}}
\put(20,0){\line(-1,1){20}}
\put(40,60){\line(-1,-1){40}}
\put(40,60){\line(1,-1){20}}
\put(20,40){\line(1,-1){20}}
\end{picture}
&\hspace*{40pt}
\begin{picture}(80,60)(0,0)
\put(40,0){\circle*{3}}
\put(38,-8){$a$}
\put(33,17){$b$}
\put(63,17){$c$}
\put(40,0){\line(1,1){20}}
\put(40,0){\line(-1,1){30}}
\put(40,0){\line(0,1){20}}
\put(40,20){\circle*{3}}
\put(60,20){\circle*{3}}
\put(60,40){\circle*{3}}
\put(40,60){\circle*{3}}
\put(10,30){\circle*{3}}
\put(60,20){\line(0,1){20}}
\put(40,20){\line(1,1){20}}
\put(40,60){\line(-1,-1){30}}
\put(40,60){\line(1,-1){20}}
\put(61,42){$b\vee c$}
\end{picture}
&\hspace*{-10pt}
\begin{picture}(80,60)(0,0)
\put(40,0){\circle*{3}}
\put(40,0){\line(1,1){30}}
\put(40,0){\line(-1,1){20}}
\put(20,20){\circle*{3}}
\put(20,40){\circle*{3}}
\put(40,40){\circle*{3}}
\put(40,60){\circle*{3}}
\put(20,20){\line(0,1){20}}
\put(20,20){\line(1,1){20}}
\put(40,60){\line(0,-1){20}}
\put(40,60){\line(-1,-1){20}}
\put(40,60){\line(1,-1){30}}
\put(70,30){\circle*{3}}
\put(14,13){$a$}
\put(13,37){$b$}
\put(43,37){$c$}
\put(30,63){$b\vee c$}
\end{picture}\end{tabular}\end{center}

In the case {\bf (c)}, we have $|L|-|L/\theta |=3$ and $L/\theta $ is a chain.

Assume by absurdum that $a/\theta =b/\theta =c/\theta $, so that $(b\vee c)/\theta =a/\theta $, therefore, since $a/\theta $ is a convex sublattice of $L$ and $|L|-|L/\theta |=3$, we have $a/\theta =\{a,b,c,b\vee c\}=[a,b\vee c]_L$, so that $b\prec b\vee c$ and $c\prec b\vee c$, and $L/\theta =\{\{a,b,c,b\vee c\}\}\cup \{\{x\}\ |\ x\in L\setminus \{a,b,c,b\vee c\}\}$. Since $L/\theta $ is a chain, it follows that, for any $x,y\in L\setminus \{a,b,c,b\vee c\}$, $\{x\}=x/\theta <a/\theta =\{a,b,c,b\vee c\}$ or $(b\vee c)/\theta =a/\theta <x/\theta $, and $x/\theta \leq y/\theta =\{y\}$ or $y/\theta \leq x/\theta $, so that $x\leq y$ or $y\leq x$, and $x<z$ for every $z\in \{a,b,c,b\vee c\}$ or $z<x$ for every $z\in \{a,b,c,b\vee c\}$. Therefore $L\cong {\cal L}_k\dotplus {\cal L}_2^2\dotplus {\cal L}_{n-k-2}$ for some $k\in \overline{1,n-3}$, so that $|{\rm Con}(L)|=2^{n-2}$, which contradicts the hypothesis that $|{\rm Con}(L)|<2^{n-2}$ of the theorem.

Therefore $a/\theta \neq c/\theta $, hence, by Remark \ref{onlgcze} and the fact that $L/\theta $ is a chain, $a/\theta \prec c/\theta =(b\vee c)/\theta $ since $a\prec c$, and, for all $x\in L\setminus (a/\theta \cup c/\theta )$ and all $y,z\in L$, we have $x\notin [a,b\vee c]_L$, either $x/\theta <a/\theta =b/\theta $ or $c/\theta =(b\vee c)/\theta <x/\theta $, and either $y/\theta \leq z/\theta $ or $z/\theta \leq y/\theta $.

Since $|L|-|L/\theta |=3>2$, we get that there exists $d\in L\setminus \{a,b,c,b\vee c\}$ such that $\{d\}\subsetneq d/\theta $. If $d\in a/\theta $, then $a/\theta $ is a three--element lattice, thus $a/\theta =\{a,b,d\}\cong {\cal L}_3$, so that $d<a<b$ or $a<b<d$ since $a\prec b$. Similarly, if $d\in c/\theta $, then $c/\theta =\{c,b\vee c,d\}\cong {\cal L}_3$, so that $d<c<b\vee c$ or $c<d<b\vee c$ or $c<b\vee c<d$. In each of these situations, the fact that $|L|-|L/\theta |=3$ shows that, for all $x,y\in L\setminus \{a,b,c,b\vee c,d\}=a/\theta \cup c/\theta $, $x/\theta =\{x\}$, hence, by the above, either $x<u$ for all $u\in \{a,b,c,b\vee c,d\}$ or $u<x$ for all $u\in \{a,b,c,b\vee c,d\}$, and either $x\leq y$ or $y\leq x$.

Assume by absurdum that $d\in a/\theta $, so that $a/\theta =\{a,b,d\}$ and $c/\theta =\{c,b\vee c\}$. Note that, in this case, $[c,b\vee c]_L=\{c,b\vee c\}$ by the convexity of $c/\theta $, thus $c\prec b\vee c$.

If $d<a<b$, then $\{d,a,b,c,b\vee c\}\cong {\cal L}_2\dotplus {\cal L}_2^2$, as in the leftmost figure below, hence $L\cong {\cal L}_k\dotplus {\cal L}_2^2\dotplus {\cal L}_{n-k-2}$ for some $k\in \overline{2,n-3}$ by the above, therefore $|{\rm Con}(L)|=2^{n-2}$, which, again, contradicts the hypothesis that $|{\rm Con}(L)|<2^{n-2}$ of the theorem.

If $a<b<d$, then $c\ngeq b<d\leq d\vee c\in (a\vee c)/\theta =c/\theta =\{c,b\vee c\}$, thus $b<d\leq d\vee c=b\vee c\neq d$, so that $b<d<b\vee c$. Therefore $\{a,b,c,d,b\vee c\}\cong N_5$, as in the second figure below, hence $L\cong {\cal L}_k\dotplus N_5\dotplus {\cal L}_{n-k-3}$ for some $k\in \overline{1,n-4}$ by the above, thus $|{\rm Con}(L)|=5\cdot 2^{n-5}$, which is a contradiction to the hypothesis that $|{\rm Con}(L)|=2^{n-3}$ of Subcase 2.1.

Now assume by absurdum that $d\in \{c,b\vee c\}$, so that $c/\theta =\{c,b\vee c,d\}$ and $a/\theta =\{a,b\}$.

If $c<b\vee c<d$, then $\{a,b,c,b\vee c,d\}\cong {\cal L}_2^2\dotplus {\cal L}_2$, as in the third figure below, hence $L\cong {\cal L}_k\dotplus {\cal L}_2^2\dotplus {\cal L}_{n-k-2}$ for some $k\in \overline{1,n-4}$ by the above, therefore $|{\rm Con}(L)|=2^{n-2}$, which, once again, contradicts the hypothesis that $|{\rm Con}(L)|<2^{n-2}$ of the theorem.

If $c<d<b\vee c$, then $\{a,b,c,d,b\vee c\}\cong N_5$, as in the fourth figure below, hence $L\cong {\cal L}_k\dotplus N_5\dotplus {\cal L}_{n-k-3}$ for some $k\in \overline{1,n-4}$ by the above, therefore $|{\rm Con}(L)|=5\cdot 2^{n-5}$, which is, again, a contradiction to the hypothesis that $|{\rm Con}(L)|=2^{n-3}$ of the current subcase.

Finally, if $d<c<b\vee c$, then $\{a,b\}=a/\theta =(a\wedge c)/\theta =(a\wedge d)/\theta $, hence $b>a\geq a\wedge d\in \{a,b\}$, thus $a\wedge d=a\neq d$, so we obtain $a<d<c$, which contradicts the fact that $a\prec c$.

Hence we are left with this possibility: $d/\theta =e/\theta $ for some $e\in L\setminus \{a,b,c,b\vee c,d\}$. Then, since $|L|-|L/\theta |=3$, it follows that $x/\theta =\{x\}$ for all $x\in L\setminus \{a,b,c,b\vee c,d,e\}$, $a/\theta =b/\theta =\{a,b\}$, $c/\theta =(b\vee c)/\theta =\{c,b\vee c\}$ and $d/\theta =e/\theta =\{d,e\}$, so that, since $c/\theta $ and $d/\theta $ are convex sublattices of $L$, thus $c/\theta \cong d/\theta \cong {\cal L}_2$, we have $c\prec b\vee c$ and either $d\prec e$ or $e\prec d$. Without loss of generality, we may assume that $d\prec e$.

$L/\theta $ is a chain and $a/\theta \prec c/\theta $, thus either $d/\theta <a/\theta \prec c/\theta $ or $a/\theta \prec c/\theta <d/\theta $. Since $a/\theta \prec c/\theta $, for any $x\in L$, $a\leq x\leq b\vee c$ implies $a/\theta \leq x/\theta \leq c/\theta $, which in turn implies $x/\theta =a/\theta $ or $x/\theta =c/\theta $, that is $x\in a/\theta \cup c/\theta =\{a,b,c,b\vee c\}$; therefore $[a,b\vee c]_L=a/\theta \cup c/\theta =\{a,b,c,b\vee c\}$, so that $b\prec b\vee c$. Therefore $b\wedge c=a\prec b\prec b\vee c$ and $a\prec c\prec b\vee c$, hence, without loss of generality, we may assume that $d/\theta <a/\theta \prec c/\theta $, because the other case is dual to this one. So, in $L/\theta $, we will have $\{d,e\}<\{a,b\}\prec \{c,b\vee c\}$. Since $L/\theta $ is a chain, we also have, for all $x\in L\setminus \{a,b,c,b\vee c,d,e\}$: either $x/\theta <a/\theta \prec c/\theta $ or $a/\theta \prec c/\theta =(b\vee c)/\theta <x/\theta $, and either $x/\theta \leq d/\theta $ or $d/\theta =e/\theta <x/\theta $, therefore, since $x/\theta =\{x\}$, we have either $x<a$ or $b\vee c<x$, and either $x<d$ or $e<x$.

If we had $e<a$, then $L\cong {\cal L}_k\dotplus {\cal L}_2^2\dotplus {\cal L}_{n-k-2}$ for some $k\in \overline{3,n-3}$, because $d,e,a,b,c,b\vee c$ would be positioned in $L$ as in the fifth figure below, thus $|{\rm Con}(L)|=2^{n-2}$, which contradicts the hypothesis that $|{\rm Con}(L)|<2^{n-2}$ of the theorem, as well as the fact that $\theta ={\rm con}(a,b)$. We have $\{d,e\}=d/\theta <a/\theta =\{a,b\}$. Since $a/\theta $ and $d/\theta =e/\theta $ are convex, we can not have $e>a$. Hence $e$ and $a$ are incomparable, $d<a$ and $e<b$. So $d\leq a\wedge e\leq e$, thus $a\wedge e\in d/\theta =\{d,e\}$ by the convexity of $d/\theta $, hence $a\wedge e=d$ since $e\not< a$ by the above. Analogously, $a\vee e=b$. Hence $\{d,e,a,b,c,b\vee c\}\cong {\cal L}_2\times {\cal L}_3$, as in the rightmost figure below.

Recall that $d\prec e$, $a\prec b\prec b\vee c$, $a\prec c\prec b\vee c$ and $[a,b\vee c]_L=\{a,b,c,b\vee c\}$. Assume by absurdum that $[d,b]_L\neq \{d,e,a,b\}$, so that $x\in [d,b]_L$ for some $x\in L\setminus \{d,e,a,b,c,b\vee c\}=L\setminus (d/\theta \cup a/\theta \cup c/\theta )$. If $x$ is comparable to neither $e$, nor $a$, then $\{d,e,x,a,b\}\cong M_3$, so that $(a,d)\in {\rm con}(a,b)=\theta $, which contradicts the fact that $a/\theta \neq d/\theta $. If $x$ is comparable to $a$, then $d<x<a$, while, if $x$ is comparable to $e$, then $e<x<b$, since $d<x<b$, $d\prec e$ and $a\prec b$; in each of these cases, $\{d,e,x,a,b\}\cong N_5$, so $x\in a/{\rm con}(a,b)=a/\theta $ in the first of these two cases, and $x\in d/{\rm con}(a,b)=d/\theta $ in the second, and each of these situations contradicts the fact that $x\notin d/\theta \cup a/\theta \cup c/\theta $. Therefore $[d,b]_L=\{d,e,a,b\}$. Hence $L\cong {\cal L}_k\dotplus ({\cal L}_2\times {\cal L}_3)\dotplus {\cal L}_{n-k-4}$ for some $k\in \overline{1,n-5}$, which, indeed, has $|{\rm Con}(L)|=2^{n-3}$.\vspace*{7pt}

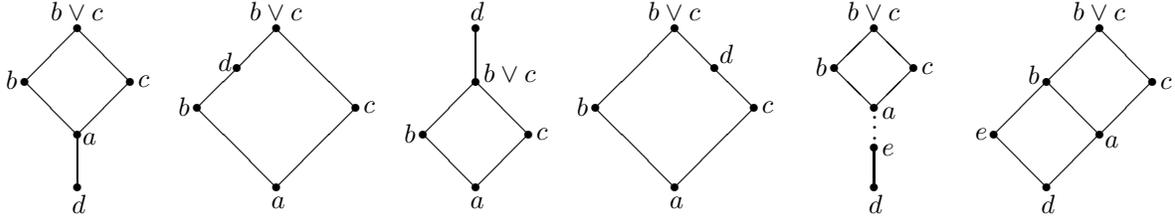
\begin{center}\begin{tabular}{cccccc}
\hspace*{-10pt}
\begin{picture}(80,60)(0,0)
\put(40,0){\circle*{3}}
\put(40,20){\circle*{3}}
\put(20,40){\circle*{3}}
\put(60,40){\circle*{3}}
\put(40,60){\circle*{3}}
\put(38,-10){$d$}
\put(30,63){$b\vee c$}
\put(13,37){$b$}
\put(63,38){$c$}
\put(42,16){$a$}
\put(40,20){\line(1,1){20}}
\put(40,20){\line(0,-1){20}}
\put(40,60){\line(-1,-1){20}}
\put(40,60){\line(1,-1){20}}
\put(20,40){\line(1,-1){20}}
\end{picture}
&\hspace*{-20pt}
\begin{picture}(80,60)(0,0)
\put(40,0){\circle*{3}}
\put(40,60){\circle*{3}}
\put(25,45){\circle*{3}}
\put(10,30){\circle*{3}}
\put(70,30){\circle*{3}}
\put(3,27){$b$}
\put(18,44){$d$}
\put(73,28){$c$}
\put(38,-8){$a$}
\put(30,63){$b\vee c$}
\put(40,60){\line(-1,-1){30}}
\put(40,60){\line(1,-1){30}}
\put(40,0){\line(-1,1){30}}
\put(40,0){\line(1,1){30}}
\end{picture}
&\hspace*{-20pt}
\begin{picture}(80,60)(0,0)
\put(40,60){\circle*{3}}
\put(40,0){\circle*{3}}
\put(20,20){\circle*{3}}
\put(60,20){\circle*{3}}
\put(40,40){\circle*{3}}
\put(38,-8){$a$}
\put(38,63){$d$}
\put(43,40){$b\vee c$}
\put(13,17){$b$}
\put(63,18){$c$}
\put(40,0){\line(1,1){20}}
\put(40,60){\line(0,-1){20}}
\put(40,40){\line(-1,-1){20}}
\put(40,40){\line(1,-1){20}}
\put(20,20){\line(1,-1){20}}
\end{picture}
&\hspace*{-20pt}
\begin{picture}(80,60)(0,0)
\put(40,0){\circle*{3}}
\put(40,60){\circle*{3}}
\put(55,45){\circle*{3}}
\put(10,30){\circle*{3}}
\put(70,30){\circle*{3}}
\put(3,27){$b$}
\put(57,47){$d$}
\put(73,28){$c$}
\put(38,-8){$a$}
\put(30,63){$b\vee c$}
\put(40,60){\line(-1,-1){30}}
\put(40,60){\line(1,-1){30}}
\put(40,0){\line(-1,1){30}}
\put(40,0){\line(1,1){30}}
\end{picture}
&\hspace*{-20pt}
\begin{picture}(80,60)(0,0)
\put(40,0){\circle*{3}}
\put(40,15){\circle*{3}}
\put(55,45){\circle*{3}}
\put(25,45){\circle*{3}}
\put(40,30){\circle*{3}}
\put(40,60){\circle*{3}}
\put(38,-10){$d$}
\put(30,63){$b\vee c$}
\put(18,42){$b$}
\put(58,43){$c$}
\put(43,26){$a$}
\put(43,12){$e$}
\put(39,18){$\vdots $}
\put(40,30){\line(1,1){15}}
\put(40,30){\line(-1,1){15}}
\put(40,60){\line(-1,-1){15}}
\put(40,60){\line(1,-1){15}}
\put(40,0){\line(0,1){15}}
\end{picture}
&\hspace*{-10pt}
\begin{picture}(80,60)(0,0)
\put(20,0){\circle*{3}}
\put(0,20){\circle*{3}}
\put(40,20){\circle*{3}}
\put(20,40){\circle*{3}}
\put(60,40){\circle*{3}}
\put(40,60){\circle*{3}}
\put(42,15){$a$}
\put(30,63){$b\vee c$}
\put(13,39){$b$}
\put(63,37){$c$}
\put(18,-10){$d$}
\put(-7,18){$e$}
\put(20,0){\line(1,1){40}}
\put(20,0){\line(-1,1){20}}
\put(40,60){\line(-1,-1){40}}
\put(40,60){\line(1,-1){20}}
\put(20,40){\line(1,-1){20}}
\end{picture}
\end{tabular}\end{center}\vspace*{2pt}

{\em Subcase 2.2:} Assume that $|{\rm Con}(L)|<2^{n-3}=8\cdot 2^{n-6}>7\cdot 2^{n-6}$ and assume by absurdum that $|{\rm Con}(L)|>7\cdot 2^{n-6}$, so that $|{\rm Con}(L/\theta )|>7\cdot 2^{n-7}>5\cdot 2^{n-7}>4\cdot 2^{n-7}=2^{n-5}$ by Lemma \ref{lgcze}, (\ref{lgcze2}), so that, by the fact that $|L/\theta |\leq n-2$, the induction hypothesis and Theorem \ref{gczetgh}, (\ref{gczetgh1}) and (\ref{gczetgh2}), we have one of the following possible situations:\begin{itemize}
\item $|L/\theta |=n-2$ and $|{\rm Con}(L/\theta )|=2^{n-3}$ or $|{\rm Con}(L/\theta )|=2^{n-4}$, but then we are in one of the situations {\bf (a)} and {\bf (b)} from Subcase $2.1$, so that, as above, $|{\rm Con}(L)|\in \{2^{n-2},2^{n-3},6\cdot 2^{n-6}\}$, and none of these values fulfills $7\cdot 2^{n-6}<|{\rm Con}(L)|<2^{n-3}$, so we have a contradiction;
\item $|L/\theta |=n-3$ and $|{\rm Con}(L/\theta )|=2^{n-4}$, so that $L/\theta \cong {\cal L}_{n-3}$, which is situation {\bf (c)} from Subcase $2.1$, in which, as above, $|{\rm Con}(L)|\in \{2^{n-2},5\cdot 2^{n-5}>4\cdot 2^{n-5}=2^{n-3},2^{n-3},2^{n-4}=4\cdot 2^{n-6}\}$, and none of these values fulfills $7\cdot 2^{n-6}<|{\rm Con}(L)|<2^{n-3}$, so we have another contradiction.\end{itemize}

Therefore $|{\rm Con}(L)|\leq 7\cdot 2^{n-6}$, which concludes the proof of the theorem.\end{proof}

\begin{remark} For all $n\in \N ^*$:\begin{itemize}
\item $(2^{n-1},n,n)$ is CFI--representable, by Theorem \ref{gczetgh}, (\ref{gczetgh1});
\item if $n\geq 4$, then $(2^{n-2},n,n)$ is CFI--representable, by Theorem \ref{gczetgh}, (\ref{gczetgh2});
\item if $n\geq 5$, then $(5\cdot 2^{n-5},n,n)$ is CFI--representable, by Theorem \ref{nextcgno}, (\ref{nextcgno1});
\item if $n\geq 6$, then $(2^{n-3},n,n)$ is CFI--representable, by Theorem \ref{nextcgno}, (\ref{nextcgno2}), and $(7\cdot 2^{n-6},n,n)$ is CFI--representable, by the case $n=6$ in the proof of Theorem \ref{nextcgno} and Remark \ref{mult2}.\end{itemize}

Of course, if $|{\rm Con}(L)|=7\cdot 2^{n-6}$, then $n\geq 6$, because otherwise $7\cdot 2^{n-6}\notin \N $. Moreover, if we also use the case $n=7$ in the proof of Theorem \ref{nextcgno}, then we obtain that, for any $r,s\in \N ^*$, ${\cal L}_r\dotplus ({\cal L}_3\boxplus {\cal L}_5)\dotplus {\cal L}_s$ and ${\cal L}_r\dotplus ({\cal L}_4\boxplus {\cal L}_4)\dotplus {\cal L}_s$ CFI--represent $(7\cdot 2^{r+s-2},r+s+4,r+s+4)$, so that, if $n\geq 6$, then, for any $k\in \overline{1,n-5}$, $L\cong {\cal L}_k\dotplus ({\cal L}_3\boxplus {\cal L}_5)\dotplus {\cal L}_{n-k-4}$ and $L\cong {\cal L}_k\dotplus ({\cal L}_4\boxplus {\cal L}_4)\dotplus {\cal L}_{n-k-4}$ CFI--represent $(7\cdot 2^{n-6},n,n)$.

We conjecture that these are the only lattices which CFI--represent $(7\cdot 2^{n-6},n,n)$, so that: $|{\rm Con}(L)|=7\cdot 2^{n-6}$ iff $n\geq 6$ and, for some $k\in \overline{1,n-5}$, $L\cong {\cal L}_k\dotplus ({\cal L}_3\boxplus {\cal L}_5)\dotplus {\cal L}_{n-k-4}$ or $L\cong {\cal L}_k\dotplus ({\cal L}_4\boxplus {\cal L}_4)\dotplus {\cal L}_{n-k-4}$.

The technique from \cite{gcze}, which we have employed in our proof of Theorem \ref{nextcgno}, is probably also adequate for this fifth largest number of congruences of an $n$--element lattice, but an induction argument for this case would greatly lengthen our paper, so we are not pursuing it in the present work.\end{remark}

\begin{corollary}\begin{enumerate}
\item\label{cglat1} If $|{\rm Con}(L)|=2^{n-1}$, then ${\rm Con}(L)\cong {\cal L}_2^{n-1}$.
\item\label{cglat2} If $|{\rm Con}(L)|=2^{n-2}$, then $n\geq 4$ and ${\rm Con}(L)\cong {\cal L}_2^{n-2}$.
\item\label{cglat3} If $|{\rm Con}(L)|=5\cdot 2^{n-5}$, then $n\geq 5$ and ${\rm Con}(L)\cong {\cal L}_2^{n-5}\times ({\cal L}_2\dotplus {\cal L}_2^2)$.
\item\label{cglat4} If $|{\rm Con}(L)|=2^{n-3}$, then $n\geq 6$ and ${\rm Con}(L)\cong {\cal L}_2^{n-3}$.\end{enumerate}\end{corollary}

\begin{proof} (\ref{cglat1}) and (\ref{cglat2}) By Theorem \ref{gczetgh} and either Remarks \ref{cgprod} and \ref{cgords}, or the fact that the lattices that CFI--represent $(2^{n-1},n,n)$ or $(2^{n-2},n,n)$ are distributive, thus modular, and, of course, finite, hence their congruence lattices are Boolean, according to \cite[Theorem $3.5.1$]{schmidt}.

\noindent (\ref{cglat3}) By Theorem \ref{nextcgno} and Remarks \ref{cgords} and \ref{cghs}. Note that the lattices that CFI--represent $(5\cdot 2^{n-5},n,n)$ are non--modular.

\noindent (\ref{cglat4}) By Theorem \ref{nextcgno} and either Remarks \ref{cgprod} and \ref{cgords}, or the fact that the lattices that CFI--represent $(2^{n-3},n,n)$ are distributive, hence, again, their congruence lattices are Boolean, according to \cite[Theorem $3.5.1$]{schmidt}.\end{proof}

\section*{Acknowledgements}

This work was supported by the research grant {\em Propriet\`a d`Ordine Nella Semantica Algebrica delle Logiche Non--classiche} of Universit\`a degli Studi di Cagliari, Regione Autonoma della Sardegna, L. R. $7/2007$, n. $7$, $2015$, CUP: ${\rm F}72{\rm F}16002920002$.

J\' ulia KULIN

kulin@math.u--szeged.hu

University of Szeged\vspace*{7pt}

Claudia MURE\c SAN

cmuresan@fmi.unibuc.ro, c.muresan@yahoo.com

University of Cagliari, University of Bucharest\end{document}